\documentclass[a4paper, 10pt, notitlepage]{article}

\usepackage{amsthm, amsmath, amssymb, latexsym}
\usepackage{mathrsfs,cite}
\usepackage[multiple]{footmisc}
\usepackage{graphicx}
\usepackage{epstopdf}
\usepackage[ruled,vlined]{algorithm2e}
\theoremstyle{plain}
\newtheorem{theorem}{Theorem}
\newtheorem{lemma}{Lemma}
\newtheorem{corollary}{Corollary}
\newtheorem{proposition}{Proposition}

\theoremstyle{definition}

\newtheorem{example}{Example}

\theoremstyle{remark}
\newtheorem{remark}{Remark}

\DeclareMathOperator{\co}{co}

\DeclareMathOperator{\dist}{dist}
\DeclareMathOperator{\diam}{diam}
\DeclareMathOperator*{\argmin}{arg\,min}
\DeclareMathOperator{\aff}{aff}

\author{M.V. Dolgopolik\footnote{Institute for Problems in Mechanical Engineering of the Russian Academy of Sciences,
Saint Petersburg, Russia}
\footnote{This work was performed in the IPME RAS and supported by the Russian Science Foundation 
(Grant No. 20-71-10032).}}
\title{An acceleration technique for methods for finding the nearest point in a polytope and computing the distance
between two polytopes}

\begin{document}

\maketitle

\begin{abstract}
We present a simple and efficient acceleration technique for an arbitrary method for computing the Euclidean projection
of a point onto a convex polytope, defined as the convex hull of a finite number of points, in the case when the number
of points in the polytope is much greater than the dimension of the space. The technique consists in applying any given
method to a ``small'' subpolytope of the original polytope and gradually shifting it, till the projection of the given
point onto the subpolytope coincides with its projection onto the original polytope. The results of numerical
experiments demonstrate the high efficiency of the proposed acceleration technique. In particular, they show that the
reduction of computation time increases with an increase of the number of points in the polytope and is proportional to
this number for some methods. In the second part of the paper, we also discuss a straightforward extension of the
proposed acceleration technique to the case of arbitrary methods for computing the distance between two convex
polytopes, defined as the convex hulls of finite sets of points.
\end{abstract}

\section{Introduction}

The problems of computing the Euclidean projection of a point onto a polytope and the distance between two polytopes are
one of central problems of computational geometry whose importance for applications cannot be overstated. A need for
fast and reliable methods for solving these problems arises in nonsmooth optimization
\cite{BagirovKarmitsaMakela,BagirovGaudioso}, submodular optimization \cite{Fujishige,Chakrabarty,FujishigeIsotani},
support vector machine algorithms \cite{KeerthiShevade,MavroforakisMargaritis,BarberoLopez,ZengYangCheng}, and many
other applications.

Since the mid 1960s, a vast array of methods for finding the nearest point in a polytope and computing the distance
between two polytopes has been developed. In the case of the nearest point problem, among them are various methods
based on quadratic programming (e.g. the Gilbert method \cite{Gilbert}), the Wolfe method \cite{Wolfe} (see also the
recent complexity analysis of a version of this method \cite{DeLoeraHaddock}), the Mitchell-Dem'yanov-Malozemov (MDM)
method \cite{MDM} and its modifications \cite{BarberoLopez,ZengYangCheng}, subgradient algorithms based on nonsmooth
penalty functions \cite{StetsyukNurminski}, geometric algorithms \cite{Muckeley,SheraliChoi}, a dual algorithm
\cite{FujishigeZhan90}, etc. Let us also mention that a number of nontrivial equivalent reformulations of the nearest
point problem was discussed in \cite{Gabidullina}.

Although the problem of computing the distance between two polytopes can be easily reduced to the nearest point problem
and the aforementioned methods can be applied to find its solution, several specialized methods for solving the distance
problem have been developed as well. A highly efficient method for computing the distance between two convex hulls in
three-dimensional space was developed by Gilbert et al. \cite{GilbertJohnsonKeerthi}, while methods for computing the
distance in the two-dimensional case were studied in \cite{Edelsbrunner,DobkinKirkpatrick,GuibasHsuZhang,YangMengMeng}.
A modification of the MDM method for computing the distance between two convex hulls in space of arbitrary dimension was
proposed \cite{Kaown,ALTMDM}, while an algorithm for solving this problem based on the fixed-point theory was studied
in \cite{LlanasEtAl}. Let us also mention a dual algorithm \cite{FujishigeZhan92}, a recursive algorithm
\cite{SekitaniYamamoto}, the sequential minimal optimization method \cite{Platt}, and some less known methods
\cite{Kozinets64,Plotnikov} for computing the distance between two polytopes.

Our main goal is to develop and analyse a general acceleration technique that can be applied to \textit{any} method for
computing the Euclidean projection of a point onto a polytope, defined as the convex hull of a finite number of points,
in the case when the number of points is significantly greater than the the dimension of the space. We present this
technique as a meta-algorithm that on each iteration employs a chosen algorithm for finding the nearest point in a
polytope as a subroutine. The acceleration technique itself consists in applying the given algorithm to a ``small''
subpolytope of the original polytope and gradually shifting it with each iteration till the required projection is
computed. 

We study both a theoretical version of the proposed acceleration technique and its robust version that is more suitable
for practical implementation, since it takes into account finite precision of computations. We prove correctness and
finite termination of both these versions under suitable assumptions, and present some very promising results of
numerical experiments. These results demonstrate a drastic reduction of computation time for several different methods
for finding the nearest point in a polytope achieved with the use of our acceleration technique. In particular,
numerical experiments show that the reduction of computation time increases with an increase of the number of points
$\ell$ in the polytope and is proportional to this number for large $\ell$.

It should be noted that the proposed acceleration technique shares many similarites with the Wolfe method \cite{Wolfe}
and the Frank-Wolfe algorithms \cite{LacosteJulienJaggi2013,LacosteJulienJaggi2015}. Nonetheless, there are important
differences between these methods related to the way in which they remove redundant points on each iteration. We
present a theoretical discussion of these differences and some results of numerical experiments demonstrating how a
different way of removing redundant points presented in this paper leads to a significant reduction of the number of
iterations (shifts of the subpolytope) in comparison with the Wolfe method.

In the second part of the paper, we extend the proposed acceleration technique to the case of \textit{arbitrary} methods
for computing the Euclidean distance between two polytopes, defined as the convex hulls of finite sets of points, in the
case when the number of points in each of these sets is significantly greater than the dimension of the space. We
present a theoretical analysis of this extension and some results of numerical experiments demonstrating the high
efficiency of the acceleration technique in the case of the distance problem.

The paper is organised as follows. Section~\ref{sect:NearestPoint} contains a detailed analysis and discussion of 
an acceleration technique for arbitrary methods for computing the Euclidean projection of a point onto a polytope.
Subsections~\ref{subsect:TheoreticalNearestPoint} and \ref{subsect:ExchangeRule} are devoted to the study of a
theoretical scheme of the acceleration technique, while Subsection~\ref{subsect:RobustNearestPoint} contains a robust
version of this technique that is more suitable for practical implementation. Some promising results of numerical
experiments for the proposed acceleration technique are collected in Subsection~\ref{subsect:NumExperiments_NPP}, while
Section~\ref{sect:ComparisonWithWolfe} contains a discussion of the differences between our acceleration technique and
the Wolfe method, as well as some results of numerical experiments highlighting these differences. Finally, a
straighforward extension of this acceleration technique to the case of method for computing the distance between two
polytopes is studied in Section~\ref{sect:DistanceBetweenPolytopes}.

\section{Finding the nearest point in a polytope}
\label{sect:NearestPoint}

In this section, we study a general acceleration technique for an \textit{arbitrary} algorithm for computing the
Euclidean projection of a point onto a polytope, defined as the convex hull of a finite number of points in
$\mathbb{R}^d$, in the case when the number of points in the polytope is much greater than the dimension of the space.
In other words, we discuss an acceleration technique for methods of solving the following optimization problem
\[
  \min_x \: \| x - z \| \enspace \text{ subject to } \enspace
  x \in P := \co\big\{ x_1, \ldots, x_{\ell} \big\} \subset \mathbb{R}^d
  \qquad \eqno{(\mathcal{P})}
\]
in the case when $\ell \gg d$. Here $z, x_1, \ldots, x_{\ell} \in \mathbb{R}^d$ are given points and $\| \cdot \|$ is
the Euclidean norm.

\subsection{General acceleration technique}
\label{subsect:TheoreticalNearestPoint}

Before we proceed to the description of the acceleration technique, let us first recall the
following well-known optimality condition for the problem $(\mathcal{P})$ (see, e.g. \cite{Wolfe,MDM}),
for the sake of completeness.

\begin{proposition} \label{prp:NearPoint_OptCond}
A point $x_* \in P$ is a globally optimal solution of the problem $(\mathcal{P})$ if and only if
\begin{equation} \label{eq:NearPoint_OptCond}
  \langle x_* - z, x_i - x_* \rangle \ge 0 \quad \forall i \in I = \{ 1, \ldots, \ell \}.
\end{equation}
\end{proposition}

Suppose that an algorithm $\mathcal{A}$ for solving the nearest point problem $(\mathcal{P})$ is given. For any point 
$w \in \mathbb{R}^d$ and any polytope $Q \subset \mathbb{R}^d$ the algorithm returns a unique solution 
$x_* = \mathcal{A}(w, Q)$ of the problem
\[
  \min_{x \in Q} \: \| x - w \|.
\]
No other assumptions on the algorithm $\mathcal{A}$ are imposed.

\begin{remark}
It should be noted that throughout this article we implicitly view all algorithms not in the way they are viewed in
computer science and the theory of algorithms, but simply as single-valued maps. In particular, the algorithm
$\mathcal{A}$ is a single-valued function mapping the Cartesian product of $\mathbb{R}^d$ and the set of all polytopes
$Q \subset \mathbb{R}^d$ into the space $\mathbb{R}^d$. It can be explicitly defined as
$\mathcal{A}(w, Q) = \argmin_{x \in Q} \| x - w \|$. One can also look at algorithm $\mathcal{A}$ as a black box with
input $(w, Q)$ and output $\argmin_{x \in Q} \| x - w \|$, which is akin to the way oracles are viewed within
optimization theory (cf. \cite{Nesterov}).
\end{remark}

Our aim is to design an acceleration technique for the algorithm $\mathcal{A}$ that would improve its performance. This
acceleration technique will be presented as a meta-algorithm that utilises the algorithm $\mathcal{A}$ as a subroutine
on each iteration.

Recall that we are only interested in the case when the number of points $\ell$ is significantly greater than the
dimension of the space $d$. Furthermore, nothing is known about the structure of the algorithm $\mathcal{A}$.
Therefore, perhaps, the only straightforward way to potentially accelerate this algorithm is by applying $\mathcal{A}$
to a polytope generated by a relatively small number of points $\{ x_{i_1}, \ldots, x_{i_s} \}$ from the set 
$\{ x_1, \ldots, x_{\ell} \}$ and then replacing some of the extreme points of the polytope 
$P_0 = \co\{ x_{i_1}, \ldots, x_{i_s} \}$ with different points from $P$ and repeating the same procedure till 
the projection of $z$ onto $P_0$ coincides with the projection of $z$ onto $P$. The key ingredient of this strategy is
\textit{an exchange rule} that replaces points from $P_0$ by different points from $P$.

Denote $\mathbb{N} = \{ 0, 1, 2, \ldots \}$ and $\mathbb{N}^* = \mathbb{N} \setminus \{ 0 \}$. To define an accelerating
meta-algorithm for solving the problem $(\mathcal{P})$, we need to choose parameter $s \in \mathbb{N}^*$ that defines
the size of the \textit{subpolytope} $P_0$ and an exchange rule $\mathcal{E}$ whose output consists of two index sets
$I_1$ and $I_2$. The first set $I_1 \subseteq \{ i_1, \ldots, i_s \}$ defines the points in $P_0$ that must be removed
from $P_0$, while the second index set $I_2 \subseteq I \setminus \{ i_1, \ldots, i_s \}$ defines the points that are
included into $P_0$ before the next iteration. In the general case, the index sets $I_1$ and $I_2$ might have arbitrary
sizes, that change with each iteration, and $I_1$ can be empty. However, for the sake of simplicity, we restrict our
consideration to the case when the cardinalities $|I_1|$ and $|I_2|$ of the sets $I_1$ and $I_2$ coincide and are equal
to some number $q \in \mathbb{N}^*$. Any exchange rule $\mathcal{E}$ that for a given input consisting of $s$ indices 
$\{ i_1, \ldots, i_s \}$ returns two collections of indices $(I_1, I_2)$ with 
\[
  I_1 \subseteq \{ i_1, \ldots, i_s \}, \quad 
  I_2 \subseteq I \setminus \{ i_1, \ldots, i_s \}, \quad
  |I_1| = |I_2| = q
\]
is called an $(s, q)$-\textit{exchange rule}.

\begin{remark}
It should be noted that an actual exchange rule obviously requires more input parameters than just a set of indices 
$\{ i_1, \ldots, i_s \}$. In particular, it might need some information about the point $z$, the set 
$\{ x_1, \ldots, x_{\ell} \}$, the nearest point $\mathcal{A}(z, P_0)$ in the polytope $P_0$, etc. as its input.
However, for the sake of shortness, below we explicitly indicate only a set of indices as an input of an exchange rule
$\mathcal{E}$.
\end{remark}

Thus, we arrive at the following scheme of the meta-algorithm for solving the problem $(\mathcal{P})$ given in
Meta-algorithm~\ref{alg:NearestPoint}. In its core, this meta-algorithm consists in choosing a \textit{subpolytope}
$P_n$ of the original polytope $P$ and finding the projection $y_n$ of the given point $z$ onto $P_n$. If this
projection happens to coincide with the projection of $z$ onto $P$ (this fact is verified with the use of optimality
conditions from Prop.~\ref{prp:NearPoint_OptCond}), then the meta-algorithm terminates. Otherwise, one replaces some
points in $P_n$, thus constructing a new subpolytope $P_{n + 1}$, and repeats the same procedure till the projection of
$z$ onto $P_n$ coincides with the projection of $z$ onto the original polytope $P$. 

From the geometric point of view, the meta-algorithm consists in choosing a ``small'' subpolytope $P_n$ in the original
polytope $P$ and gradually ``shifting'' it with each iteration till $P_n$ contains the projection of $z$ onto $P$. This
procedure is performed with the hope that it is much faster to compute a projection of a point onto a small subpolytope
and gradually shift it, rather than to compute the projection of this point onto the original polytope with a very large
number of vertices (i.e. with $\ell \gg d$). The results of numerical experiments reported below demonstrate that this
hope is fully justified.

\begin{algorithm}[t]  \label{alg:NearestPoint}
\caption{Meta-algorithm for finding the nearest point in a polytope.}

\noindent\textbf{Input:} {a point $z \in \mathbb{R}^d$, a collection of points $\{ x_1, \ldots, x_{\ell} \} \subset
\mathbb{R}^d$, an algorithm $\mathcal{A}$ for solving the nearest point problem, parameters 
$s, q \in \{ 1, \ldots, \ell \}$ with $s \ge q$, and an $(s, q)$-exchange rule $\mathcal{E}$.}

\noindent\textbf{Initialization:} {Put $n = 0$, choose an index set $I_n \subseteq I$ with $|I_n| = s$, and define 
$P_n = \co\{ x_i \mid i \in I_n \}$.}

\noindent\textbf{Step 1:} {Compute $y_n = \mathcal{A}(z, P_n)$. If $y_n$ satisfies the optimality condition
\begin{equation} \label{eq:TrialPointOptCond}
  \langle y_n - z, x_i - y_n \rangle \ge 0 \quad \forall i \in I,
\end{equation}
\textbf{return} $y_n$.}

\noindent\textbf{Step 2:} {Compute $(I_{1n}, I_{2n}) = \mathcal{E}(I_n)$ and define
\[
  I_{n + 1} = \Big( I_n \setminus I_{1n} \Big) \cup I_{2n}, \quad 
  P_{n + 1} = \co\big\{ x_i \bigm| i \in I_{n + 1} \big\}.
\]
Set $n = n + 1$ and go to \textbf{Step 1}.}
\end{algorithm}

The two following lemmas describe simple conditions on the exchange rule $\mathcal{E}$ ensuring that the proposed
meta-algorithm indeed solves the problem $(\mathcal{P})$ and, furthermore, terminates after a finite number of
iterations. These lemmas provide one with two convenient criteria for choosing effective exchange rules. Although
(rather awkward) proofs of these lemmas are obvious, we present them for the sake of completeness.

For any set $\Omega \subset \mathbb{R}^d$ and any $x \in \mathbb{R}^d$ denote by 
$\dist(x, \Omega) = \inf_{y \in \Omega} \| x - y \|$ the distance from $x$ to $\Omega$.

\begin{lemma} \label{lem:DistDecay}
Suppose that the exchange rule $\mathcal{E}$ satisfies \textbf{the distance decay condition}: if for some 
$n \in \mathbb{N}$ the point $y_n$ does not satisfy optimality condition \eqref{eq:TrialPointOptCond}, then
\begin{equation} \label{eq:DistDecayDef}
  \dist(z, P_{n + 1}) < \dist(z, P_n).
\end{equation}
Then Meta-algorithm~\ref{alg:NearestPoint} terminates after a finite number of steps and returns an optimal solution of
the problem $(\mathcal{P})$.
\end{lemma}

\begin{proof}
From Proposition~\ref{prp:NearPoint_OptCond} and the termination criterion \eqref{eq:TrialPointOptCond} of
Meta-algorithm~\ref{alg:NearestPoint} it follows that if this algorithm terminates in a finite number of steps, then the
last computed point $y_n$ (and the output of Meta-algorithm~\ref{alg:NearestPoint}) is an optimal solution of the
problem $(\mathcal{P})$. Thus, we only need to prove that the algorithm terminates in a finite number of steps.

To prove the finite termination, note that there is only a finite number of distinct subsets of the set 
$\{ x_1, \ldots, x_{\ell} \}$ with cardinality $s$. Moreover, from the distance decay condition 
\eqref{eq:DistDecayDef} it follows that if the algorithm does not terminate in $n \in \mathbb{N}$
iterations, then all index sets $I_0, I_1, \ldots, I_n$ are distinct. Therefore, in a finite number of steps the
algorithm must find a point $y_n$ satisfying the termination criterion \eqref{eq:TrialPointOptCond}.
\end{proof}

\begin{remark}
If the exchange rule $\mathcal{E}$ satisfies the distance decay condition, then Meta-algorithm~\ref{alg:NearestPoint}
generates a finite sequence of polytopes $\{ P_n \} \subset P$ such that $P_n \ne P_k$ for any $n \ne k$. Note that
the length of such sequence does not exceed the number of $s$-combinations of the set $\{ 1, \ldots, \ell \}$, which is
equal to $\binom{\ell}{s} = \mathcal{O}(\ell^s)$. Therefore, Meta-algorithm~\ref{alg:NearestPoint} has polynomial in
$\ell$ complexity, provided the exchange rule $\mathcal{E}$ satisfies the distance decay condition and has polynomial
in $\ell$ complexity as well.
\end{remark}

\begin{lemma} \label{lem:SubpolytopeSize}
Let $\ell \ge d + 1$ and an $(s, q)$-exchange rule $\mathcal{E}$ with $s \ge q$ satisfy the distance decay condition for
any point $z \in \mathbb{R}^d$ and any polytope $P = \co\{ x_1, \ldots, x_{\ell} \}$. Then $s \ge d + 1$.
\end{lemma}

\begin{proof}
Suppose by contradiction that $s \le d$. Let us provide a particular point $z$ and a particular polytope $P$ for which
any $(s, q)$-exchange rule with $q \le s \le d$fails to satisfy the distance decay condition.

Let $z = 0$. If $d = 1$, define $\ell = 2$, $x_1 = 1$, and $x_2 = -1$. In this case $s = q = 1$ and for any choice of
$I_0$ it is obviously impossible to satisfy the condition $\dist(z, P_1) < \dist(z, P_2)$.

Let now $d \ge 2$. Put $z = 0$, $\ell = d + 1$, and define the points $x_i$ as follows:
\begin{gather*}
  x_1 = (1, 0, \ldots, 0, -1), \: x_2 = (0, 1, 0, \ldots, 0, -1), \ldots, \:
  x_{d - 2} = (0, \ldots, 0, 1, 0, -1),
  \\
  x_{d - 1} = (-1, \ldots, -1, 1, -1), \quad x_d = (-1, \ldots, -1), \quad x_{d + 1} = (0, \ldots, 0, 1).
\end{gather*}
As is easily seen, any $d$ points from the set $\{ x_1, \ldots, x_{d + 1} \}$ are linearly independent and, in addition,
$0 \in P$, since
\[
  \sum_{i = 1}^{d + 1} \alpha_i x_i = 0, \quad \sum_{i = 1}^{d + 1} \alpha_i = 1
\]
for
\[
  \alpha_1 = \ldots = \alpha_{d - 2} = \frac{1}{2(d - 1)}, \quad \alpha_{d - 1} = \alpha_d = \frac{1}{4(d - 1)}, 
  \quad \alpha_{d + 1} = \frac{1}{2}.
\]
Hence, in particular, for any index set $K \subset I$ with $|K| = s \le d$ one has $\dist(z, P_K) > 0$, where 
$P_K = \co\{ x_i \mid i \in K \}$. Therefore, for any choice of an $(s, q)$-exchange rule with $s \le d$ and $s \ge q$
the stopping criterion \eqref{eq:TrialPointOptCond} cannot be satisfied, which by the previous lemma implies that this
exchange rule does not satisfy the distance decay condition.
\end{proof}

\begin{remark}
One can readily verify that if in the lemma above one imposes the additional assumption that $z \notin P$, then the
statement of the lemma holds true for $s \ge d$. To prove this result, one simply needs to put $z = 0$, $\ell = d$, and
define $P$ as the convex hull of the first $d$ points from the proof of the lemma above.
\end{remark}

\subsection{The steepest descent exchange rule}
\label{subsect:ExchangeRule}

Let us present a detailed analysis of a particular exchange rule based on the optimality condition from 
Prop.~\ref{prp:NearPoint_OptCond} (or, equivalently, the stopping criterion \eqref{eq:TrialPointOptCond}) and
satisfying the distance decay condition for any point $z$ and any polytope $P$.

Bearing in mind Lemma~\ref{lem:SubpolytopeSize}, we propose to consider the following $(d + 1, 1)$-exchange rule that
on each iteration of Meta-algorithm~\ref{alg:NearestPoint} replaces only one point in the current polytope $P_n$.
Suppose that for some $n \in \mathbb{N}$ the stopping criterion \eqref{eq:TrialPointOptCond} is not satisfied. Then we
define the new point $x_{2n}$ from $P$, that is included into $P_{n + 1}$, as any point from $P$ on which the minimum in
\[
  \min_{i \in I} \langle y_n - z, x_i - y_n \rangle
\]
is attained (cf. a similar rule for including new points in the major cycle of the Wolfe method \cite{Wolfe}). 

To find a point that is removed from $P_n$, note that $z \notin P_n$ due to the definition of $y_n$ and the fact that
condition \eqref{eq:TrialPointOptCond} does not hold true. Therefore, $y_n$ belongs to the boundary of $P_n$, which by
\cite[Lemma~2.8]{Ziegler} implies that $y_n$ is contained in a face of $P_n$ of dimension at most $d - 1$.
Hence by \cite[Prop.~1.15 and 2.3]{Ziegler} the point $y_n$ can be represented as a convex combination of at
most $d$ points from the set $\{ x_i \mid i \in I_n \}$. In other words, there exists $i_{1n} \in I_n$ such that
$y_n \in \co\{ x_i \mid I_n \setminus \{ i_{1n} \} \}$, and it is natural to remove the point $x_{i_{1n}}$
from the polytope $P_n$. 

Thus, we arrive at the following theoretical scheme of a $(d + 1, 1)$-exchange rule that we call 
\textit{the steepest descent exchange rule}:
\begin{itemize}
\item{\textbf{Input:} an index set $I_n \subset I$ with $|I_n| = d + 1$, the point $z$, 
the set $\{ x_1, \ldots, x_{\ell} \}$, and the projection $y_n$ of $z$ onto $P_n = \co\{ x_i \mid i \in I_n \}$.
}

\item{\textbf{Step 1:} Find $i_{1n} \in I_n$ such that $y_n \in \co\{ x_i \mid I_n \setminus \{ i_{1n} \} \}$.}

\item{\textbf{Step 2:} Find $i_{2n} \in I$ such that
\[
  \langle y_n - z, x_{i_{2n}} \rangle = \min_{i \in I} \langle y_n - z, x_i \rangle.
\]
\textbf{Return} $(\{ i_{1n} \}, \{ i_{2n} \})$.}
\end{itemize}

From the discussion above it follows that the steepest descent exchange rule is correctly defined, provided $y_n$ does
not satisfy the stopping criterion \eqref{eq:TrialPointOptCond}. Let us verify that the steepest descent exchange rule
always satisfies the distance decay condition. The proof of this result is almost the same as the proof of the analogous
property for the iterates of the Wolfe method \cite{Wolfe,Chakrabarty}. We include a full proof of this result for the
sake of completeness and due to the fact that Meta-algorithm \ref{alg:NearestPoint} with the steepest descent exchange
rule, strictly speaking, does not coincide with the Wolfe method (see Section~\ref{sect:ComparisonWithWolfe} for more
details).

\begin{theorem} \label{thm:StDescExchRule}
For any point $z \in \mathbb{R}^d$ and any polytope $P = \co\{ x_1, \ldots, x_{\ell} \}$ with $\ell \ge d + 1$ 
the steepest descent exchange rule satisfies the distance decay condition.
\end{theorem}

\begin{proof}
Suppose that for some $n \in \mathbb{N}$ the stopping criterion \eqref{eq:TrialPointOptCond} does not hold true. Denote
by 
\[
  P_n^0 = \co\{ x_i \mid i \in I_n \setminus \{ i_{1n} \} \}
\]
the polytope obtained from $P_n$ after removing the point selected by the steepest descent exchange rule. By
construction $y_n \in P_n^0$. Moreover, due to the definition of the exchange rule and the fact that the stopping
criterion is not satisfied one has 
\begin{equation} \label{eq:GradAtNewPoint}
  \langle y_n - z, x_{i_{2n}} - y_n \rangle < 0.
\end{equation}
Define $x_n(t) = (1 - t) y_n + t x_{i_{2n}}$. Clearly, $x_n(t) \in P_{n + 1}$ for any
$t \in [0, 1]$, since $P_{n + 1} = \co\{ P_n^0, x_{i_{2n}} \}$. Moreover, for any $t \in \mathbb{R}$ one has
\begin{align*}
  f(t) &:= \| x_n(t) - z \|^2 
  \\
  &= (1 - t)^2 \| y_n - z \|^2 + 2 t (1 - t) \langle y_n - z, x_{i_{2n}} - z \rangle + t^2 \| x_{i_{2n}} - z \|^2.
\end{align*}
Note that $f(0) = \| y_n - z \|^2$ and
\[
  f'(0) = - 2 \| y_n - z \|^2 + 2 \langle y_n - z, x_{i_{2n}} - z \rangle 
  = 2\langle y_n - z, x_{i_{2n}} - y_n \rangle < 0
\]
due to \eqref{eq:GradAtNewPoint}. Therefore, for any sufficiently small $t \in (0, 1)$ one has
$f(t) < f(0)$, which implies that
\begin{align*}
  \dist(z, P_{n + 1}) &= \min_{x \in P_{n + 1}} \| z - x \| \le \| z - x_n(t) \| 
  \\
  &= \sqrt{f(t)} < \sqrt{f(0)} = \| z - y_n \| = \dist(z, P_n).
\end{align*}
Thus, the steepest descent exchange rule satisfies the distance decay condition for any point $z \in \mathbb{R}^d$ and
any polytope $P = \co\{ x_1, \ldots, x_{\ell} \}$ with $\ell \ge d + 1$.
\end{proof}

\begin{corollary}
Let $\ell \ge d + 1$. Then Meta-algorithm~\ref{alg:NearestPoint} with $s = d + 1$, $q = 1$, and the steepest descent
exchange rule terminates after a finite number of iterations and returns an optimal solution of the problem
$(\mathcal{P})$.
\end{corollary}

Let us discuss a possible implementation of the steepest descent exchange rule. Clearly, the challenging part of this
rule consists in finding an index $i_{1n} \in I_n$ such that $y_n \in \co\{ x_i \mid I_n \setminus \{ i_{1n} \} \}$.
This difficulty can be overcome in the following way.

Namely, suppose that instead of returning the projection $u_* = \mathcal{A}(w, Q)$ of a point $w$ onto a polytope 
$Q = \co\{ u_1, \ldots, u_m \}$, the algorithm $\mathcal{A}$ actually returns a vector 
$\alpha = (\alpha^{(1)}, \ldots, \alpha^{(m)}) \in \mathbb{R}^{m}$ such that
\begin{equation} \label{eq:ConvexCombinOutput}
  u_* = \sum_{i = 1}^m \alpha^{(i)} u_i, \quad \sum_{i = 1}^m \alpha^{(i)} = 1, \quad \alpha_i \ge 0
  \quad \forall i \in \{ 1, \ldots, m \}.
\end{equation}
Let us note that for the vast majority of existing methods for finding the nearest point in a polytope this assumption
either holds true by default or can be satisfied by slightly modifying the corresponding method (see
\cite{Wolfe,MDM,BarberoLopez,ZengYangCheng,StetsyukNurminski,Muckeley,SheraliChoi}).

Let $\alpha_n = \mathcal{A}(z, P_n)$ and suppose that $\alpha_n^{(k)} = 0$ for some $k \in I_n$. Then one can obviously
set $i_{1n} = k$ on Step~1 of the steepest descent exchange rule. To simplify the notation, hereinafter we identify 
the vector $\alpha_n = \mathcal{A}(z, P_n) \in \mathbb{R}^{d + 1}$ with the extended vector 
$\widehat{\alpha}_n \in \mathbb{R}^{\ell}$ such that $\widehat{\alpha}_n^{(i)} = \alpha_n^{(i)}$ for any 
$i \in I_n$, and $\widehat{\alpha}_n^{(i)} = 0$ otherwise.

It should be noted that in the case when the points $x_i$, $i \in I_n$, are affinely independent and the stopping
criterion \eqref{eq:TrialPointOptCond} is not satisfied, there always exists $k \in I_n$ such that 
$\alpha_n^{(k)} = 0$. Indeed, as was noted above, in this case $y_n$ does not belong to the interior of $P_n$, which by 
the characterization of interior points of a polytope \cite[Lemma~2.8]{Ziegler} implies that $y_n$ \textit{cannot} be
represented in the form
\[
  y_n = \sum_{i \in I_n} \alpha^{(i)} x_i, \quad \sum_{i \in I_n} \alpha^{(i)} = 1, 
  \quad \alpha^{(i)} > 0 \quad \forall i \in I_n.
\]
Consequently, at least one of the coordinates of the vector $\alpha_n = \mathcal{A}(z, P_n)$ is equal to zero. 

Even when the points $x_i$, $i \in I_n$, are affinely dependent, some methods (such as the the Wolfe method
\cite{Wolfe}) necessarily return a vector $\alpha_n = \mathcal{A}(z, P_n)$ with $\alpha_n^{(k)} = 0$ for at least one 
$k \in I_n$. However, other methods might return a vector $\alpha_n$ such that $\alpha_n^{(i)} > 0$ for all $i \in I_n$.
In this case one can apply the following simple procedure, inspired the the proof of Carath\'{e}odory's theorem, to find
the required index $i_{1n}$. This procedure is described in algorithm below and we call it 
\textit{the index removal method}:
\begin{itemize}
\item{\textbf{Input:} an index set $I_n \subset I$ with $|I_n| = d + 1$, the point $z$, 
the set $\{ x_1, \ldots, x_{\ell} \}$, and $\alpha_n = \mathcal{A}(z, P_n)$.
}

\item{\textbf{Step 1:} Compute $\alpha_{\min} = \min_{i \in I_n} \alpha_n^{(i)}$. If $\alpha_{\min} = 0$,
find $k \in I_n$ such that $\alpha_n^{(k)} = 0$ and \textbf{return} $k$.
}

\item{\textbf{Step 2:} Choose any $j \in I_n$, compute a least-squares solution $\gamma_n$ of the system
\begin{equation} \label{eq:AffDependenceSystem}
  \sum_{i \in I_n \setminus \{ j \}} \gamma^{(i)} (x_i - x_j) = 0, \quad 
  \sum_{i \in I_n \setminus \{ j \}} \gamma^{(i)} = 1,
\end{equation}
and set $\gamma_n^{(j)} = - 1$. Find an index $k \in I_n$ on which the minimum in
\begin{equation} \label{eq:CaratheodoryCoefficient}
  \min\left\{ - \frac{\alpha_n^{(k)}}{\gamma_n^{(k)}} \Biggm| 
  k \in I_n \colon \gamma_n^{(i)} < 0 \right\}
\end{equation}
is attained and \textbf{return} $k$.
}
\end{itemize}

The following proposition proves the correctness of the proposed method.

\begin{proposition} \label{prp:IndexRemovalMethod}
Suppose that for some $n \in \mathbb{N}$ the stopping criterion \eqref{eq:TrialPointOptCond} does not hold true, and
let $k \in I_n$ be the output of the index removal method. 
Then $y_n \in \co\{ x_i \mid i \in I_n \setminus \{ k \} \}$. 
\end{proposition}

\begin{proof}
The validity of the proposition in the case $\alpha_{\min} = 0$ is obvious. Therefore, let us consider the case
$\alpha_{\min} > 0$, that is, the case when the index removal method executes Step~2. As was pointed out above, in this
case the points $x_i$, $i \in I_n$, are necessarily affinely dependent, which by definition implies that the vectors 
$x_i - x_j$, $i \in I_n \setminus \{ j \}$, are linearly dependent. Therefore system of linear equations
\eqref{eq:AffDependenceSystem} is consistent (despite being overdetermined) and its least-squares solution $\gamma_n$ 
satisfies equations \eqref{eq:AffDependenceSystem}.

By definition $\gamma_n \ne 0$ and 
\begin{equation} \label{eq:AffineDependence}
  \sum_{i \in I_n} \gamma_n^{(i)} x_i = 0, \quad \sum_{i \in I_n} \gamma_n^{(i)} = 0.
\end{equation}
Moreover, $\gamma_n^{(j)} = - 1$. Consiquently, the minimum in \eqref{eq:CaratheodoryCoefficient}, which we denote by
$\lambda$, is correctly defined and $\lambda > 0$ (recall that $\alpha_{\min} > 0$).

Observe that
\[
  y_n = \sum_{i \in I_n} \alpha_n^{(i)} x_i 
  = \sum_{i \in I_n} \alpha_n^{(i)} x_i + \lambda \sum_{i \in I_n} \gamma_n^{(i)} x_i
  = \sum_{i \in I_n} \xi_n^{(i)} x_i,
\]
where $\xi_n^{(i)} = \alpha_n^{(i)} + \lambda \gamma_n^{(i)}$, by the first equality in \eqref{eq:AffineDependence}.
By the second equality in \eqref{eq:AffineDependence} and the definition of $\alpha_n$ (see 
\eqref{eq:ConvexCombinOutput}) one has $\sum_{i \in I_n} \xi_n^{(i)} = 1$. Moreover, if $\gamma_n^{(i)} \ge 0$,
then clearly $\xi_n^{(i)} > 0$, while if $\gamma_n^{(i)} < 0$, then $\xi_n^{(i)} \ge 0$ by the definition of
$\lambda$ (see \eqref{eq:CaratheodoryCoefficient}). In addition, $\xi_n^{(k)} = 0$ by the definition of $k$, which
obviously implies that $y_n \in \co\{ x_i \mid i \in I_n \setminus \{ k \} \}$. 
\end{proof}

\begin{remark}
Let us point out that the second equation in \eqref{eq:CaratheodoryCoefficient} can obviously be replaced by 
the equation $\sum_{i \in I_n \setminus \{ j \}} \gamma^{(i)} = C$ for any $C > 0$ (and one also has to set
$\gamma_n^{(j)} = - C$).
\end{remark}

\begin{remark}
It should be noted that in the general case the steepest descent exchange rule does not preserve the affine independence
of the vectors $x_i$, $i \in I_n$, as the following simple example demonstrates. Let $d = 2$, $\ell = 4$, $z = 0$, and
\[
  x_1 = (2, 2), \quad x_2 = (3, 1), \quad x_3 = (1, 1), \quad x_4 = (-1, 1).
\]
Put $I_0 = \{ 1, 2, 3 \}$. Then $y_0 = x_3$, $\alpha_0 = \mathcal{A}(z, P_0) = (0, 0, 1)$, and the stopping criterion
\eqref{eq:TrialPointOptCond} is not satisfied. One can set $i_{10} = 1$, while by definition $i_{20} = 4$. Thus,
$I_1 = \{ 2, 3, 4 \}$. The points $x_i$, $i \in I_1$, are obviously affinely dependent, while the points $x_i$, 
$i \in I_0$, are affinely independent. Thus, the difficulty of finding the required index $i_{1n} \in I_n$ in the case
when $\alpha_n^{(i)} > 0$ for all $i \in I_n$ (i.e. when the points $x_i$, $i \in I_n$, are affinely dependent) cannot
be resolved by simply choosing an initial guess $I_0$ in such a way that the points $x_i$, $i \in I_0$, are affinely
independent. 
\end{remark}

\subsection{A robust version of the meta-algorithm}
\label{subsect:RobustNearestPoint}

It is clear that any algorithm $\mathcal{A}$ for solving the problem $(\mathcal{P})$ can find an optimal solution of
this problem only in theory, while in practice it always returns an approximate solution of this problem due to 
finite precision of computations. Therefore it is an important practical issue to analyse the performance of
Meta-algorithm~\ref{alg:NearestPoint} in the case when only computations with finite precision are possible.

Assume that instead of the ``ideal'' algorithm $\mathcal{A}$ one uses its ``approximate'' version 
$\mathcal{A}_{\varepsilon}$, $\varepsilon > 0$. The algorithm $\mathcal{A}_{\varepsilon}$ returns an approximate, in
some sense (that will be specified below), solution $y_n(\varepsilon)$ of the nearest point problem. Then it is obvious
that the stopping criterion  
\[
  \langle y_n(\varepsilon) - z, x_i - y_n(\varepsilon) \rangle \ge 0 \quad \forall i \in I
\]
of Meta-algorithm~\ref{alg:NearestPoint} cannot be satisfied and must be replaced by the inequality
\[
  \langle y_n(\varepsilon) - z, x_i - y_n(\varepsilon) \rangle \ge - \eta \quad \forall i \in I
\]
for some small $\eta > 0$. The following well-known result (cf.~\cite{MDM}) indicates a direct connection between 
the inequality above and approximate optimality of $y_n(\varepsilon)$. For the sake of completeness, we include a full
proof of this result.

Denote $\diam(P) = \max_{i, j \in I} \| x_i - x_j \|$. One can readily verify that the inequality 
$\| x - y \| \le \diam(P)$ holds true for all $x, y \in P$, which means that $\diam(P)$ is indeed the diameter of the
polytope $P$.

\begin{proposition} \label{prp:ApproximateOptimality}
Let $y \in P$ satisfy the inequality
\begin{equation} \label{eq:SubOptimCond}
  \langle y - z, x_i - y \rangle \ge - \eta \quad \forall i \in I
\end{equation}
for some $\eta > 0$. Then $\| y - x_* \| \le \sqrt{\eta}$, where $x_*$ is an optimal solution of the problem
$(\mathcal{P})$. 

Conversely, let $y \in P$ be such that $\| y - x_* \| \le \varepsilon$ for some $\varepsilon > 0$. Then $y$
satisfies inequality \eqref{eq:SubOptimCond} for any $\eta \ge (\diam(P) + \dist(z, P)) \varepsilon$.
\end{proposition}

\begin{proof}
Let a point $y \in P$ satisfy inequality \eqref{eq:SubOptimCond} for some $\eta > 0$. Adding and subtracting $z$ one
gets
\begin{align*}
  \| y - x_* \|^2 &= \| y - z \|^2 - 2 \langle y - z, x_* - z \rangle + \| x_* - z \|^2
  \\
  &= \Big( \| y - z \|^2 - \langle y - z, x_* - z \rangle \Big) 
  + \Big( \| x_* - z \|^2 - \langle y - z, x_* - z \rangle \Big)
  \\
  &= \langle y - z, y - x_* \rangle - \langle x_* - z, y - x_* \rangle. 
\end{align*}
Hence applying Prop.~\ref{prp:NearPoint_OptCond} one obtains
\[
  \| y - x_* \|^2 \le \langle y - z, y - x_* \rangle.
\] 
Since $x_* \in P$, there exist $\alpha_i \ge 0$, $i \in I$, such that $x_* = \sum_{i \in I} \alpha_i x_i$ and
$\sum_{i \in I} \alpha_i = 1$. Therefore, with the use of \eqref{eq:SubOptimCond} one finally gets that
\[
  \| y - x_* \|^2 \le \sum_{i = 1}^{\ell} \alpha_i \langle y - z, y - x_i \rangle
  \le \sum_{i = 1}^{\ell} \alpha_i \eta = \eta,
\]
which implies that $\| y - x_* \| \le \sqrt{\eta}$.

Suppose now that $\| y - x_* \| \le \varepsilon$ for some $\varepsilon > 0$. Adding and subtracting $x_*$ one has
\begin{align*}
  \langle y - z, x_i - y \rangle &= \langle x_* - z, x_i - y \rangle + \langle y - x_*, x_i - y \rangle
  \\
  &= \langle x_* - z, x_i - x_* \rangle + \langle x_* - z, x_* - y \rangle + \langle y - x_*, x_i - y \rangle.
\end{align*}
Hence with the use of Prop.~\ref{prp:NearPoint_OptCond} and the definition of $x_*$ one gets that
\begin{align*}
  \langle y - z, x_i - y \rangle &\ge - \Big( \| x_* - z \| + \| x_i - y \| \Big) \| x_* - y \|
  \\
  &\ge - \Big( \dist(z, P) + \diam(P) \Big) \varepsilon,
\end{align*}
which implies the required result.
\end{proof}

Although a robust version of Meta-algorithm~\ref{alg:NearestPoint} can be formulated for an arbitrary exchange rule,
for the sake of brevity we will formulate it only in the case of the steepest descent exchange rule. To this end, we
suppose that the output of the approximate algorithm $\mathcal{A}_{\varepsilon}(z, P_n)$ is not an approximate solution
$y_n(\varepsilon)$ of the nearest point problem $\min_{x \in P_n} \| x - z \|$, but a vector $\alpha_n(\varepsilon)$ of
coefficients of the corresponding convex combination, that is,
\[
  y_n(\varepsilon) = \sum_{i \in I_n} \alpha_n^{(i)}(\varepsilon) x_i, \quad
  \sum_{i \in I_n} \alpha_n^{(i)}(\varepsilon) = 1, \quad 
  \alpha_n^{(i)}(\varepsilon) \ge 0 \quad \forall i \in I_n.
\]
Due to finite precision of computations, even in the case when the vectors $x_i$, $i \in I_n$, are affinely
independent, all coefficients $\alpha_n^{(i)}(\varepsilon)$, $i \in I_n$, might be nonzero. Therefore we propose to use 
the following heuristic rule for choosing a point $x_{i_{1n}}$, $i_{1n} \in I_n$, that is removed from the polytope
$P_n$ by the steepest descent exchange rule. Namely, we choose as $i_{1n}$ any index from $I_n$ on which the minimum in
$\min_{i \in I_n} \alpha_n^{(i)}(\varepsilon)$ is attained and add a safeguard based on the index removal method,
discussed above, to ensure the validity of a certain approximate distance decay condition. Note, however, that for some
methods (such as the Wolfe method \cite{Wolfe}) there always exists $i \in I_n$ such that 
$\alpha_n^{(i)}(\varepsilon) = 0$, even when the computations are performed with finite precision. For such methods 
the safeguard based on the index removal method is completely unnecessary.

\begin{algorithm}[htbp]  \label{alg:NearestPointRobust}
\caption{Robust meta-algorithm for finding the nearest point in a polytope.}

\noindent\textbf{Input:} {a point $z \in \mathbb{R}^d$, a collection of points 
$\{ x_1, \ldots, x_{\ell} \} \subset \mathbb{R}^d$, an algorithm $\mathcal{A}_{\varepsilon}$, $\varepsilon > 0$, for
solving the nearest point problem, and $\eta > 0$.}

\noindent\textbf{Step 0:} {Put $n = 0$, choose an index set $I_n \subseteq I$ with $|I_n| = d + 1$, and define 
$P_n = \co\{ x_i \mid i \in I_n \}$. Compute $\alpha_n(\varepsilon) = \mathcal{A}_{\varepsilon}(z, P_n)$,
$y_n(\varepsilon) = \sum_{i \in I_n} \alpha_n^{(i)}(\varepsilon) x_i$, and $\theta_n = \| y_n(\varepsilon) - z \|$.
}

\noindent{\textbf{Step 1:} If $y_n(\varepsilon)$ satisfies the condition
\begin{equation} \label{eq:TrialPointApproxOptCond}
  \langle y_n(\varepsilon) - z, x_i - y_n(\varepsilon) \rangle \ge -\eta \quad \forall i \in I,
\end{equation}
\textbf{return} $y_n(\varepsilon)$. Otherwise, find $i_{1n} \in I_n$ and $i_{2n} \in I$ such that
\[
  \alpha_n^{(i_{1n})}(\varepsilon) = \min_{i \in I_n} \alpha_n^{(i)}(\varepsilon), \quad
  \langle y_n(\varepsilon) - z, x_{i_{2n}} \rangle = 
  \min_{i \in I} \langle y_n(\varepsilon) - z, x_i \rangle.
\]
Put 
\[
  I_{n + 1} = \Big( I_n \setminus \{ i_{1n} \} \Big) \cup \{ i_{2n} \}, \quad
  P_{n + 1} = \co\Big\{ x_i \Bigm| i \in I_{n + 1} \Big\}.
\]
}

\noindent\textbf{Step 2:} {Compute $\alpha_{n + 1}(\varepsilon) = \mathcal{A}_{\varepsilon}(z, P_{n + 1})$,
$y_{n + 1}(\varepsilon) = \sum_{i \in I_{n + 1}} \alpha_{n + 1}^{(i)}(\varepsilon) x_i$, and 
$\theta_{n + 1} = \| y_{n + 1}(\varepsilon) - z \|$. If $\theta_{n + 1} < \theta_n$, set $n = n + 1$ and go to
\textbf{Step 1}. Otherwise, go to \textbf{Step 3}.
}

\noindent{\textbf{Step 3:} Find an approximate solution $\beta_n$ of the problem
\begin{equation} \label{eq:AuxProjectionProblem}
  \min_{\beta} \Big\| \sum_{i \in I_n} \beta^{(i)} x_i - z \Big \|^2 
  \quad \text{subject to} \quad \sum_{i \in I_n} \beta^{(i)} = 1.
\end{equation}
Compute $h_n = \sum_{i \in I_n} \beta_n^{(i)} x_i$ and $\beta_{\min} = \min_{i \in I_n} \beta_n^{(i)}$. 
If $\beta_{\min} < 0$, find
\begin{equation} \label{eq:WolfeCoef}
  \lambda = \min\left\{ \frac{\alpha_n^{(i)}}{\alpha_n^{(i)} - \beta_n^{(i)}} \Biggm| 
  i \in I_n \colon \beta_n^{(i)} < 0 \right\},
\end{equation}
define $\alpha_n(\varepsilon) = (1 - \lambda) \alpha_n(\varepsilon) + \lambda \beta_n$, 
$y_n(\varepsilon) = (1 - \lambda) y_n(\varepsilon) + \lambda h_n$, and $\theta_n = \| y_n(\varepsilon) - z \|$,
and go to \textbf{Step~1}. If $\beta_{\min} = 0$, define $\alpha_n(\varepsilon) = \beta_n$, $y_n(\varepsilon) = h_n$,
$\theta_n = \| h_n - z \|$, and go to \textbf{Step~1}. If $\beta_{\min} > 0$, go to \textbf{Step 4}.
}

\noindent{\textbf{Step 4:} Choose any $j \in I_n$, find an approximate least-squares solution $\gamma_n$ of
the system
\[
  \sum_{i \in I_n \setminus \{ j \}} \gamma^{(i)} (x_i - x_j) = 0, \quad 
  \sum_{i \in I_n \setminus \{ j \}} \gamma^{(i)} = 1,
\]
and set $\gamma_n^{(j)} = - \sum_{i \in I_n \setminus \{ j \}} \gamma_n^{(i)}$. Compute
\[
  \lambda = \min\left\{ - \frac{ \alpha_n^{(i)}}{\gamma_n^{(i)}} \Biggm| i \in I_n \colon \gamma_n^{(i)} < 0 \right\}.
\]
Define $\alpha_n(\varepsilon) = \alpha_n(\varepsilon) + \lambda \gamma_n$,
$y_n(\varepsilon) = \sum_{i \in I_n} \alpha_n(\varepsilon) x_i$, $\theta_n = \| y_n(\varepsilon) - z \|$,
and go to \textbf{Step 1}.
}
\end{algorithm}

Thus, we arrive at the following robust version of Meta-algorithm~\ref{alg:NearestPoint} given in
Meta-algorithm~\ref{alg:NearestPointRobust}. This meta-algorithm checks the approximate distance decay condition
\[
  \| y_{n + 1}(\varepsilon) - z \| < \| y_n(\varepsilon) - z \|
\]
to verify the correctness of the index exchange. If the condition fails, one needs to rectify the choice of the index
$i_{1n}$ on the previous step (that is, a wrong point was removed from $P_n$ and one must remove a different point).

To correct the choice of $i_{1n}$, the meta-algorithm first computes the projection of $z$ onto the affine hull 
$\aff\{ x_i \mid i \in I_n \}$ of the points $x_i$, $i \in I_n$, on Step 3 (see problem
\eqref{eq:AuxProjectionProblem}). Let us note that problem \eqref{eq:AuxProjectionProblem} can be reduced to the problem
of solving a system of linear equations (see \cite{Wolfe}).

As will be shown below, if the points $x_i$, $i \in I_n$, are affinely independent or the projection of $z$ onto their
affine hull does not coincide with the projection of $z$ onto their convex hull, Step 3 makes a necessary correction
of the point $y_n(\varepsilon)$ (more precisely, the coefficients $\alpha_n(\varepsilon)$ of the corresponding convex
combination) to ensure that the new choice of $i_{1n}$ on Step~1 leads to the validity of the approximate distance decay
condition. Otherwise, the meta-algorithm moves to Step 4 and employs essentially the same technique as in the index
removal method to correct the coefficients $\alpha_n(\varepsilon)$ and find the required index $i_{1n}$. 

Let us analyze Meta-algorithm~\ref{alg:NearestPointRobust}. First, we show that if
$\min_{i \in I_n} \alpha_n^{(i)}(\varepsilon) = 0$ for some $n \in \mathbb{N}$ on Step~1 of this meta-algorithm, then
under some natural assumptions no corrections of the coefficients $\alpha_n(\varepsilon)$ are needed, the method
does not execute Steps~3 and 4, and moves to the next (i.e. $(n + 1)$th) iteration.

As in the previous section, for any $n \in \mathbb{N}$ denote by $y_n$ the actual projection of $z$ onto $P_n$, that
is, an optimal solution of the problem $\min_{y \in P_n} \| y - z \|$

\begin{lemma} \label{lem:ApproxDistanceDecay}
Suppose that
\begin{multline} \label{eq:ParametersConsistency}
  \max\Big\{ 2(\diam(P) + \dist(z, P)) \varepsilon, \\
  \diam(P) \sqrt{\max\{ 0, 2 \varepsilon \theta_0 - \varepsilon^2\}} \Big\} < \eta
  \le \diam(P)^2
\end{multline}
and the algorithm $\mathcal{A}_{\varepsilon}$ with  $\varepsilon \ge 0$ satisfies the following \textbf{approximate
optimality condition}: for any point $w \in \mathbb{R}^d$ and any polytope 
$Q = \co\{ u_1, \ldots, u_m \} \subset \mathbb{R}^d$ one has 
\[
  \Big\| \sum_{i = 1}^m \alpha^{(i)} u_i - u_* \Big\| < \varepsilon, \quad
  \sum_{i = 1}^m \alpha^{(i)} = 1, \quad \alpha^{(i)} \ge 0 \quad \forall i \in \{ 1, \ldots, m \},
\]
where $\alpha = \mathcal{A}_{\varepsilon}(w, Q)$ and $u_*$ is an optimal solution of the nearest point problem 
$\min_{u \in Q} \| u - w \|$. Let also for some $n \in \mathbb{N}$ one has $\theta_{k + 1} < \theta_k$ for any 
$k \in \{ 0, 1, \ldots, n - 1 \}$ and
\begin{equation} \label{eq:MinZeroConvexCombin}
  \min_{i \in I_n} \alpha_n^{(i)}(\varepsilon) = 0
\end{equation}
on Step~1 of Meta-algorithm~\ref{alg:NearestPointRobust}. Then for $\theta_{n + 1}$ computed on Step~2 of
Meta-algorithm~\ref{alg:NearestPointRobust} one has $\theta_{n + 1} < \theta_n$.
\end{lemma}

\begin{proof}
Let us divide the proof into two parts. First, we show that $\theta_k \ge \varepsilon$ for any 
$k \in \{ 0, \ldots, n \}$ and then use this result to prove the statement of the lemma.

\textbf{Part~1}. Suppose by contradiction that $\theta_k < \varepsilon$ for some $k \in \{ 0, 1, \ldots, n \}$, that
is, $\| y_k(\varepsilon) - z \| < \varepsilon$. Let $x_*$ be an optimal solution of the problem $(\mathcal{P})$. Then by
definition $\| x_* - z \| \le \| y_k(\varepsilon) - z \| < \varepsilon$, which yields 
$\| x_* - y_k(\varepsilon) \| < 2 \varepsilon$. Hence by Proposition~\ref{prp:ApproximateOptimality} one has
\[
  \langle y_k(\varepsilon) - z, x_i - y_k(\varepsilon) \rangle 
  \ge - 2(\diam(P) + \dist(z, P)) \varepsilon
\]
Therefore by the first inequality in \eqref{eq:ParametersConsistency} the point $y_k(\varepsilon)$ satisfies the
stopping criterion \eqref{eq:TrialPointApproxOptCond}, which contradicts our assumption that the meta-algorithm
computes $i_{1n}$ on Step~1 of the $n$th iteration for $n \ge k$.

\textbf{Part~2}. By our assumption 
$\alpha_n^{(i_{1n})}(\varepsilon) = \min_{i \in I_n} \alpha_n^{(i)}(\varepsilon) = 0$. Therefore 
\[
  y_n(\varepsilon) = \sum_{i \in I_n} \alpha_n^{(i)}(\varepsilon) x_i \in 
  \co\Big\{ x_i \Bigm| i \in I_n \setminus \{ i_{1n} \} \Big\},
\]
which yields $y_n(\varepsilon) \in P_{n + 1}$, where the set $P_{n + 1}$ is defined on Step~1.

By the definition of $i_{2n}$ (see Step~1 of the meta-algorithm) one has
\begin{equation} \label{eq:NotApproxOptimal}
  \langle y_n(\varepsilon) - z, x_{i_{2n}} - y_n(\varepsilon) \rangle 
  = \min_{i \in I} \langle y_n(\varepsilon) - z, x_i - y_n(\varepsilon) \rangle \le - \eta.
\end{equation}
Define $x_n(t) = (1 - t) y_n(\varepsilon) + t x_{i_{2n}}$. Clearly, $x_n(t) \in P_{n + 1}$ for all $t \in [0, 1]$,
since, as was noted above, $y_n(\varepsilon) \in P_{n + 1}$ and $x_{i_{2n}} \in P_{n + 1}$ by the definition of 
$P_{n + 1}$ (see Step~2 of the meta-algorithm). 

For any $t \in \mathbb{R}$ one has
\begin{align*}
  f(t) &:= \| x_n(t) - z \|^2 
  = \Big\| (y_n(\varepsilon) - z) + t (x_{i_{2n}} - y_n(\varepsilon)) \Big\|^2
  \\
  &= \| y_n(\varepsilon) - z \|^2 + 2 t \langle y_n(\varepsilon) - z, x_{i_{2n}} - y_n(\varepsilon) \rangle
  + t^2 \| x_{i_{2n}} - y_n(\varepsilon) \|^2.
\end{align*}
Hence with the use of \eqref{eq:NotApproxOptimal} one obtains
\[
  f(t) \le \theta_n^2 - 2 \eta t + \diam(P)^2 t^2 \quad \forall \alpha \ge 0.
\]
The minimum in $t$ of the right-hand side of this inequality is attained at $t_* = \eta / \diam(P)^2$. From 
the second inequality in \eqref{eq:ParametersConsistency} it follows that $t_* \in (0, 1]$. Therefore, the point
$x_n(t_*)$ belongs to $P_{n + 1}$ and
\[
  \min_{t \in [0, 1]} f(t) = f(t_*) 
  = \| x_n(t_*) - z \|^2 \le \theta_n^2 - \frac{\eta^2}{\diam(P)^2}.
\]
Applying the definition of $y_{n + 1}$, the first inequality in \eqref{eq:ParametersConsistency}, and the fact that
$\theta_0 \ge \varepsilon$ one gets
\[
  \| y_{n + 1} - z \|^2 = \min_{y \in P_{n + 1}} \| y - z \|
  \le \| x_n(t_*) - z \|^2 < \theta_n^2 - 2 \theta_0 \varepsilon + \varepsilon^2.
\]
Hence $\| y_{n + 1} - z \|^2 < (\theta_n - \varepsilon)^2$, thanks to the fact that
$\varepsilon \le \theta_n < \theta_{n - 1} < \ldots < \theta_0$ by our assumption and the first part of the proof.
Consequently, by the approximate optimality condition on $\mathcal{A}_{\varepsilon}$, for $y_{n + 1}(\varepsilon)$
computed on Step~2 one has
\[
  \theta_{n + 1} = \| y_{n + 1}(\varepsilon) - z \| \le \| y_{n + 1} - z \| + \varepsilon
  < \theta_n - \varepsilon + \varepsilon = \theta_n.
\]
Thus, the approximate distance decay holds true, the meta-algorithm increments $n$ on Step~2 and does not execute
Steps~3 and 4.
\end{proof}

\begin{remark}
Let us underline that the lemma above holds true regardless of whether $\theta_0, \theta_1, \ldots, \theta_n$, and
$\alpha_n(\varepsilon)$ were computed on Step 2, 3 or 4. In particular, it holds true even if the equality
\eqref{eq:MinZeroConvexCombin} is satisfied for $\alpha_n(\varepsilon)$ that was computed on Steps 3 or 4 and
not directly computed by the algorithm $\mathcal{A}_{\varepsilon}$.
\end{remark}

\begin{remark}
Note that the value $\theta_0$ that is a priori unknown is used in inequalities \eqref{eq:ParametersConsistency} on
parameters of Meta-algorithm~\ref{alg:NearestPointRobust}. However, we can easily estimate it from above.
If $\theta_0$ is computed on Step~0, then by the approximate optimality condition 
$\theta_0 \le \dist(z, P_0) + \varepsilon$. In turn, if $\theta_0$ is computed on Steps~3 or 4, then under some natural
assumptions one can show that $\theta_0 \le \dist(z, P_0) + 2 \varepsilon$ (see the proof of
Theorem~\ref{thm:RobustMetaAlgCorrectness} below).
\end{remark}

The previous lemma allows one to immediately prove correctness and finite termination of
Meta-algorithm~\ref{alg:NearestPointRobust} in the case when the algorithm $\mathcal{A}_{\varepsilon}$ always returns a
vector $\alpha_n(\varepsilon)$ having at least one zero component, regardless of whether the points $x_i$, $i \in I_n$,
are affinely independent or not. Recall that this assumption is satisfied for the Wolfe method \cite{Wolfe}.

\begin{theorem} \label{thrm:RobustMetaAlg_Wolfe}
Let $\ell \ge d + 1$, inequalities \eqref{eq:ParametersConsistency} hold true, and the algorithm
$\mathcal{A}_{\varepsilon}$ with $\varepsilon > 0$ satisfy the approximate optimality condition from
Lemma~\ref{lem:ApproxDistanceDecay}. Suppose also that for any point $w \in \mathbb{R}^d$ and any polytope 
$Q = \co\{ u_1, \ldots, u_m \} \subset \mathbb{R}^d$ with $m \ge d + 1$ there
exists $i \in \{ 1, \ldots, m \}$ such that for $\alpha = \mathcal{A}_{\varepsilon}(w, Q)$ one has $\alpha^{(i)} = 0$.
Then Meta-algorithm~\ref{alg:NearestPointRobust} is correctly defined, never executes Steps 3 and 4, terminates after 
a finite number of iterations, and returns a point $y_n(\varepsilon)$ such that 
$\| y_n(\varepsilon) - x_* \| < \sqrt{\eta}$, where $x_*$ is an optimal solution of the problem $(\mathcal{P})$.
\end{theorem}

\begin{proof}
From the assumptions of the theorem and Lemma~\ref{lem:ApproxDistanceDecay} it follows that $\theta_1 < \theta_0$ and
the meta-algorithm does not execute Steps~3 and 4 for $n = 0$, provided the stopping criterion
\eqref{eq:TrialPointApproxOptCond} is not satisfied for $y_0(\varepsilon)$ (otherwise, the method terminates when 
$n = 0$ and does not execute Step~2). 

Now, arguing by induction and applying Lemma~\ref{lem:ApproxDistanceDecay} one can check that 
$\theta_n < \theta_{n - 1}$ and the meta-algorithm does not execute Steps~3 and 4 for any $n \in \mathbb{N}$, if
the stopping criterion \eqref{eq:TrialPointApproxOptCond} is not satisfied for $y_k(\varepsilon)$ with
$k \in \{ 0, 1, \ldots, n - 1 \}$ (otherwise, the meta-algorithm never reaches $n$th iteration). 

Thus, the meta-algorithm is correctly defined and the corresponding (finite or infinite) sequence 
$\{ y_n(\varepsilon) \}$ satisfies the approximate distance decay condition: 
\[
  \| y_n(\varepsilon) - z \| = \theta_n < \theta_{n - 1} = \| y_{n - 1}(\varepsilon) - z \|.
\]
From this inequality it follows that all polytopes $P_0, P_1, \ldots, P_n, \ldots$ are distinct. Recall that each
$P_i$ is the convex hull of $d + 1$ points from the set $\{ x_1, \ldots, x_{\ell} \}$. Since there is only a finite
number of distinct $d + 1$-point subsets of the set $\{ x_1, \ldots, x_{\ell} \}$, one must conclude that after a finite
number of iterations the stopping criterion \eqref{eq:TrialPointApproxOptCond} must be satisfied, that is, the
meta-algorithm terminates after a finite number of iterations. Moreover, any point $y_n(\varepsilon)$ satisfying 
the stopping criterion also satisfies the inequality $\| y_n(\varepsilon) - x_* \| < \sqrt{\eta}$ by
Proposition~\ref{prp:ApproximateOptimality}. 
\end{proof}

\begin{remark}
Let us comment on the assumption \eqref{eq:ParametersConsistency} on parameters of the algorithm
$\mathcal{A}_{\varepsilon}$ and Meta-algorithm~\ref{alg:NearestPointRobust}. Roughly speaking, inequality
\eqref{eq:ParametersConsistency} means that to solve the problem $(\mathcal{P})$ with a pre-specified accuracy 
$\eta > 0$ with the use of Meta-algorithm~\ref{alg:NearestPointRobust} one must assume that the algorithm
$\mathcal{A}_{\varepsilon}$ solves the corresponding reduced nearest point subproblems $\min_{x \in P_n} \| x - z \|$
with higher accuracy. Qualitatively, condition \eqref{eq:ParametersConsistency} can be rewritten as 
$\eta = O(\sqrt{\varepsilon})$ and viewed as a mathematical formulation of an intuitively obvious fact that
Meta-algorithm~\ref{alg:NearestPointRobust} has lower accuracy than the algorithm $\mathcal{A}_{\varepsilon}$ that is
used as a subroutine on each iteration. However, when considered quantitatively, inequalities
\eqref{eq:ParametersConsistency} seem to be too conservative. They can be slighly relaxed, if one uses a different
stopping criterion of the form
\[
  \langle y_n(\varepsilon) - z, x_i - y_n(\varepsilon) \rangle \ge -\eta \| x_i - y_n(\varepsilon) \|
  \quad \forall i \in I.
\]
Then arguing in the same way as in the proof of Lemma~\ref{lem:ApproxDistanceDecay} one can check that it is sufficient
to suppose that $\sqrt{2 \varepsilon \theta_0 + \varepsilon^2} < \eta \le 1$.
\end{remark}

\begin{remark}
Note that one can replace the approximate optimality condition on the algorithm $\mathcal{A}_{\varepsilon}$ from 
Lemma~\ref{lem:ApproxDistanceDecay} by the following condition: for any point $w \in \mathbb{R}^d$ and any polytope 
$Q = \co\{ u_1, \ldots, u_m \} \subset \mathbb{R}^d$ the point $y = \sum_{i = 1}^m \alpha^{(i)} u_i$ with
$\alpha = \mathcal{A}_{\varepsilon}(w, Q)$ satisfies the inequality
\[
  \langle y - z, u_i - y \rangle \ge - \varepsilon \quad \forall i \in \{ 1, \ldots, m \},
\]
that is, the algorithm $\mathcal{A}_{\varepsilon}$ returns a point approximately satisfying the optimality conditions
for the problem $\min_{u \in Q} \| u - w \|$. If $u_*$ is an optimal solution of this problem, then by
Proposition~\ref{prp:ApproximateOptimality} one has 
$\| \mathcal{A}_{\varepsilon}(w, Q) - u_* \| \le \sqrt{\varepsilon}$. Consequently, the theorem above remains to hold
true in this case, provided $\varepsilon$ is replaced by $\sqrt{\varepsilon}$ in the first inequality in
\eqref{eq:ParametersConsistency}.
\end{remark}

Let us now prove correctness and finite termination of Meta-algorithm~\ref{alg:NearestPointRobust} in the general case.

\begin{theorem} \label{thm:RobustMetaAlgCorrectness}
Let $\ell \ge d + 1$, inequalities \eqref{eq:ParametersConsistency} be satisfied, and the following 
approximate optimality conditions hold true:
\begin{enumerate}
\item{for any point $w \in \mathbb{R}^d$ and any polytope $Q = \co\{ u_1, \ldots, u_m \} \subset \mathbb{R}^d$ one has 
\[
  \Big\| \sum_{i = 1}^m \alpha^{(i)} u_i - u_* \Big\| < \varepsilon, \quad
  \sum_{i = 1}^m \alpha^{(i)} = 1, \quad \alpha^{(i)} \ge 0 \quad \forall i \in \{ 1, \ldots, m \},
\]
where $\alpha = \mathcal{A}_{\varepsilon}(w, Q)$ and $u_*$ is an optimal solution of the nearest point problem 
$\min_{u \in Q} \| u - w \|$;
}

\item{if for some $n \in \mathbb{N}$ the meta-algorithm executes Step 3, then $\sum_{i \in I_n} \beta_n^{(i)} = 1$ and
$\| h_n - z \| \le \| y_n(\varepsilon) - z \|$; 
}

\item{if for some $n \in \mathbb{N}$ the vectors $x_i$, $i \in I_n$, are affinely independent and 
the meta-algorithm executes Step~3, then $\beta_{\min} \le 0$;
}

\item{if for some $n \in \mathbb{N}$ the meta-algorithm executes Step~4, then
\[
  \Big\| \sum_{i \in I_n \setminus \{ j \}} \gamma^{(i)} (x_i - x_j) \Big\| 
  < \frac{\theta_{n - 1} - \theta_n}{\lambda}, \quad 
  \sum_{i \in I_n \setminus \{ j \}} \gamma^{(i)} \ne 0
\]
where $\lambda > 0$ is computed on Step~4 (if $n = 0$, then only the second inequality should be satisfied).
}
\end{enumerate}
Then Meta-algorithm~\ref{alg:NearestPointRobust} is correctly defined, executes Steps~3 and 4 at most once per
iteration, terminates after a finite number of iterations, and returns a point $y_n(\varepsilon)$ such that 
$\| y_n(\varepsilon) - x_* \| < \sqrt{\eta}$, where $x_*$ is an optimal solution of the problem $(\mathcal{P})$.
\end{theorem}

\begin{proof}
Firstly, let us note that if the following two conditions hold true:
\begin{enumerate}
\item{the meta-algorithm is correctly defined (that is, there are no infinite loops involving Steps 3 and 4),}

\item{the updating of $\theta_n$ on Steps~3 and 4 preserves the condition $\theta_n < \theta_{n - 1}$,}
\end{enumerate}
then the meta-algorithm generates a finite or infinite sequence $\{ y_n(\varepsilon) \}$ satisfying the approximate
distance decay condition:
\begin{equation} \label{eq:ApproxDistDecayOverSeq}
  \theta_0 > \theta_1 > \ldots > \theta_n > \ldots
\end{equation}
(recall that $\theta_n = \| y_n(\varepsilon) - z \|$). Indeed, according to the description of the method (see
Meta-algorithm~\ref{alg:NearestPointRobust}), the meta-algorithm increments $n$ and moves to the next iteration if and
only if the condition $\theta_{n + 1} < \theta_n$ is satisfied on Step~2. Otherwise, it moves to Steps 3 and 4, corrects
$\alpha_n(\varepsilon)$ and $y_n(\varepsilon)$, and repeats Steps 1 and 2 till the condition $\theta_{n + 1} < \theta_n$
is satisfied. Thus, if (i) there are no infinite loops involving Steps 3 and 4, and (ii) the condition 
$\theta_n < \theta_{n - 1}$ is preserved after updating $\theta_n$ on Steps~3 and 4, then inequalities
\eqref{eq:ApproxDistDecayOverSeq} hold true.

Secondly, as was noted in the proof of Theorem~\ref{thrm:RobustMetaAlg_Wolfe}, the validity of the approximate distance
decay condition \eqref{eq:ApproxDistDecayOverSeq} implies that all polytopes $P_0, P_1, \ldots, P_n, \ldots$ are
distinct. Hence taking into account the facts that each $P_i$ is the convex hull of $d + 1$ points from the set 
$\{ x_1, \ldots, x_{\ell} \}$ and there is only a finite number of distinct $d + 1$-point subsets of
$\{ x_1, \ldots, x_{\ell} \}$, one must conclude that after a finite number of iterations the stopping criterion
\eqref{eq:TrialPointApproxOptCond} must be satisfied. Moreover, any point $y_n(\varepsilon)$ satisfying this criterion
also satisfies the inequality $\| y_n(\varepsilon) - x_* \| < \sqrt{\eta}$ by
Proposition~\ref{prp:ApproximateOptimality}. 

Thus, to complete the proof of the theorem we need to prove that (i) there are no infinite loops involving Steps~3 and
4, and (ii) the updating of $\theta_n$ on Steps~3 and 4 does not break the condition $\theta_n < \theta_{n - 1}$. In
addition, our aim is to prove a slightly stronger statement that Steps 3 and 4 are executed at most once per iteration.
Let us prove all these statements by induction.

Since the proof of the case $n = 0$ is essentially the same as the proof of the inductive step, we will consider only
the inductive step. Fix any $n \in \mathbb{N}^*$ and suppose that for any $k \in \{ 0, 1, \ldots, n - 1 \}$ Steps~3 and
4 were performed at most once during the $k$th iteration and the condition 
\begin{equation} \label{eq:ApproxDistDecay_nthIter}
  \theta_0 > \theta_1 > \ldots > \theta_{n - 2} > \theta_{n - 1} >\theta_n
\end{equation}
holds true.

Clearly, we only need to consider the case when the meta-algorithm executes Step~3 on iteration $n$ and Step~3 has not
been executed before on this iteration. In this case, according to the scheme of the meta-algorithm
$y_n(\varepsilon) = \sum_{i \in I_n} \alpha_n^{(i)}(\varepsilon) x_i$ with
$\alpha_n(\varepsilon) = \mathcal{A}_{\varepsilon}(z, P_n)$ and the point $y_n(\varepsilon)$ does not satisfy the
stopping criterion \eqref{eq:TrialPointApproxOptCond} (see Steps~1--4). Let us check that $z \notin P_n$.

Indeed, suppose that $z \in P_n$. Then $y_n = z$ is an optimal solution of the problem $\min_{y \in P} \| y - z \|$.
Observe that by the first approximate optimality condition one has $\| y_n - y_n(\varepsilon) \| \le \varepsilon$.
Therefore by Proposition~\ref{prp:ApproximateOptimality} one has
\[
  \langle y_n(\varepsilon) - z, x_i - y_n(\varepsilon) \rangle \ge - \diam(P_n) \varepsilon
  \ge - \diam(P) \varepsilon
  \quad \forall i \in I.
\]
Hence taking into account the first inequality in \eqref{eq:ParametersConsistency} one can conclude that 
$y_n(\varepsilon)$ satisfies the stopping criterion \eqref{eq:TrialPointApproxOptCond}, which is impossible. 

Thus, $z \notin P_n$. Note also that $\min_{i \in I_n} \alpha_n(\varepsilon) > 0$, since otherwise by 
Lemma~\ref{lem:ApproxDistanceDecay} the meta-algorithm does not execute Step~3.

Recall that $h_n = \sum_{i \in I_n} \beta_n^{(i)} x_i$ is the approximate projection of $z$ onto the affine hull
$\aff\{ x_i \mid i \in I_n \}$ computed on Step~3 and $\beta_{\min} = \min_{i \in I_n} \beta_n^{(i)}$. Let us consider
three cases.

\textbf{Case I.} Let $\beta_{\min} < 0$. Observe that for $\lambda$ defined in \eqref{eq:WolfeCoef} one has
$\lambda \in (0, 1)$, since $\alpha_n^{(i)}(\varepsilon) > 0$ for all $i \in I_n$. Define 
$\xi_n = (1 - \lambda)\alpha_n(\varepsilon) + \lambda \beta_n$ and $w_n = (1 - \lambda) y_n(\varepsilon) + \lambda h_n$.
Then
\[
  \sum_{i \in I_n} \xi_n^{(i)} = 
  (1 - \lambda) \sum_{i \in I_n} \alpha_n^{(i)}(\varepsilon) + \lambda \sum_{i \in I_n} \beta_n^{(i)} 
  = (1 - \lambda) + \lambda = 1
\]
(here the penultimate equality holds true by the first and second approximate optimality conditions). Moreover, if
$\beta_n^{(i)} \ge 0$, then obviously $\xi_n^{(i)} \ge 0$, while if $\beta_n^{(i)} < 0$, then
\[
  \xi_n^{(i)} = (\alpha_n^{(i)}(\varepsilon) - \beta_n^{(i)}) 
  \left( \frac{\alpha_n^{(i)}(\varepsilon)}{\alpha_n^{(i)}(\varepsilon) - \beta_n^{(i)}} - \lambda \right) \ge 0
\]
by the definition of $\lambda$. In addition, for any $i$ on which the minimum in the definition of $\lambda$ is
attained (see \eqref{eq:WolfeCoef}) one has $\xi_n^{(i)} = 0$. Hence the point 
$w_n = (1 - \lambda) y_n(\varepsilon) + \lambda h_n$ is a convex combination of the vectors $x_i$, $i \in I_n$, that is,
$w_n \in P_n$. Furthermore, by the second approximate optimality condition one has
\[
  \| w_n - z \| \le (1 - \lambda) \| y_n(\varepsilon) - z \| + \lambda \| h_n - z \|
  \le \| y_n(\varepsilon) - z \| = \theta_n.
\]
Thus, after updating $\alpha_n(\varepsilon)$, $y_n(\varepsilon)$, and $\theta_n$ on Step~3 one has
\begin{equation} \label{eq:UpdateCorrectness}
  \sum_{i \in I_n} \alpha_n^{(i)}(\varepsilon) = 1, \quad \min_{i \in I_n} \alpha_n^{(i)}(\varepsilon) = 0,
  \quad \theta_n < \theta_{n - 1}.
\end{equation}
Consequently, by Lemma~\ref{lem:ApproxDistanceDecay} after performing Step~1 and computing
$\alpha_{n + 1}(\varepsilon) = \mathcal{A}_{\varepsilon}(z, P_{n + 1})$ on Step~2 one has $\theta_{n + 1} < \theta_n$.
Thus, the meta-algorithm increments $n$ and moves to the next iteration. In other words, in the case $\beta_{\min} < 0$
Step~3 of the meta-algorithm is performed only once and the condition $\theta_n < \theta_{n - 1}$ is preserved.

\textbf{Case II.} Let $\beta_{\min} = 0$. Then $h_n \in P_n$. Moreover, by the second optimality condition
$\| h_n - z \| \le \| y_n(\varepsilon) - z \|$, which implies that after setting $\theta_n = \| h_n - z \|$ 
the condition $\theta_n < \theta_{n - 1}$ is preserved. In addition, for the updated value of 
$\alpha_n(\varepsilon) = \beta_n(\varepsilon)$ one has $\min_{i \in I_n} \alpha_n^{(i)}(\varepsilon) = 0$. Therefore,
by Lemma~\ref{lem:ApproxDistanceDecay} one can conclude that on Step~2 the condition $\theta_{n + 1} < \theta_n$ is
satisfied. Thus, in the case $\beta_{\min} = 0$ the meta-algorithm executes Step~3 only once and then after performing
Steps~1 and 2 moves to the next iteration.

\textbf{Case III.} Let $\beta_{\min} > 0$. In this case the points $x_i$, $i \in I_n$, are affinely dependent by
the third approximate optimality condition, and the meta-algorithm moves to Step~4. Let $\gamma_n$ and $\lambda$ be
computed on Step~4. Recall that by definitions
\[
  \lambda 
  = \min\left\{ - \frac{ \alpha_n^{(i)}}{\gamma_n^{(i)}} \Biggm| i \in I_n \colon \gamma_n^{(i)} < 0 \right\} > 0.
\]
and $\sum_{i \in I_n} \gamma_n^{(i)} = 0$. 

Define $\xi_n = \alpha_n(\varepsilon) + \lambda \gamma_n$. Then $\sum_{i \in I_n} \xi_n^{(i)} = 1$, for any $i \in I_n$
such that $\gamma_n^{(i)} \ge 0$ one has $\xi_n^{(i)} \ge 0$, while for any $i \in I_n$ such that $\gamma_n^{(i)} < 0$
one has
\[
  \xi_n^{(i)} = - \gamma_n^{(i)} \left( -\frac{\alpha_n^{(i)}(\varepsilon)}{\gamma_n^{(i)}} - \lambda \right) \ge 0.
\]
In addition, for any $i \in I_n$ on which the minimum in the definition of $\lambda$ is attained one has
$\xi_n^{(i)} = 0$. Note finally that by the fourth approximate optimality condition
\begin{multline*}
  \Big\| \sum_{i \in I_n} \xi_n^{(i)} x_i - z \Big\|
  = \Big\| \sum_{i \in I_n} \alpha_n^{(i)}(\varepsilon) x_i - z 
  + \lambda \sum_{i \ne j} \gamma_n^{(i)} (x_i - x_j) \Big\|
  \\
  \le \| y_n(\varepsilon) - z \| + \lambda \Big\| \sum_{i \ne j} \gamma_n^{(i)} (x_i - x_j) \Big\|
  < \theta_n + \lambda \frac{\theta_{n - 1} - \theta_n}{\lambda} = \theta_{n - 1}.
\end{multline*}
Therefore, after an update on Step~4, values of $\alpha_n(\varepsilon)$, $y_n(\varepsilon)$, and $\theta_n$ satisfy
conditions \eqref{eq:UpdateCorrectness}, which thanks to Lemma~\ref{lem:ApproxDistanceDecay} implies that after
performing Steps~1 and 2 the meta-algorithm increments $n$ and moves to the next iteration. In other words, in the case
$\beta_{\min} > 0$ the meta-algorithm executes Steps~3 and 4 only once and then moves to the next iteration.
Furthermore, the condition $\theta_n < \theta_{n - 1}$ is preserved in this case as well.
\end{proof}

\begin{remark}
Let us comment on the approximate optimality conditions from the previous theorem:

\noindent{(i)~The first condition simply states that the algorithm $\mathcal{A}_{\varepsilon}(w, Q)$ always returns 
a point from $Q$ that lies in the $\varepsilon$ neighbourhood of the projection of $w$ onto $Q$. In turn, the fourth
condition indicates the accuracy with which an approximate least squares solution on Step~4 should be computed. Note
that while the second equality in
\[
  \sum_{i \in I_n \setminus \{ j \}} \gamma^{(i)} (x_i - x_j) = 0, \quad 
  \sum_{i \in I_n \setminus \{ j \}} \gamma^{(i)} = 1
\]
is essentially irrelevant, as long as the sum of $\gamma^{(i)}$, $i \in I_n \setminus \{ j \}$, is nonzero, the first
equality must be solved with high enough accuracy to ensure that after updating $y_n(\varepsilon)$ the inequality
$\theta_n < \theta_{n - 1}$ still holds true.
}

\noindent{(ii)~The second approximate optimality condition states that the approximate distance to the affine hull
$x_i$, $i \in I_n$, computed by a subroutine on Step~3 of Meta-algorithm~\ref{alg:NearestPointRobust}, does not exceed
the approximate distance to the convex hull of these points computed by the algorithm $\mathcal{A}_{\varepsilon}$.
Roughly speaking, the second approximate optimality condition means that an approximate projection of $z$ onto the
affine hull $\aff\{ x_i \mid i \in I_n \}$, is computed on Step~3 with at least the same accuracy as the algorithm
$\mathcal{A}_{\varepsilon}$ computes an approximate projection of $z$ onto the convex hull 
$P_n = \co\{ x_n \mid i \in I_n \}$.
}

\noindent{(iii)~The third approximate optimality condition is needed to exclude some degenerate cases. It should be
noted that in the ideal case when $\varepsilon = 0$, this assumption is not needed. Indeed, if the points $x_i$, 
$i \in I_n$, are affinely independent, then their affine hull coincides with the entire space $\mathbb{R}^d$.
Hence taking into account the fact that $z \notin P_n$ (otherwise, the stopping criterion
\eqref{eq:TrialPointApproxOptCond} would have been satisfied) one can conclude that $\beta_{\min} < 0$. However, when
computations are performed with finite precision, for some highly degenerate problems one might have $\beta_{\min} > 0$,
even in the case when $x_i$, $i \in I_n$, are affinely independent, due to computational errors. In such cases the
method might get stuck in an infinite loop of correcting the coefficients $\alpha_n(\varepsilon)$. The third approximate
optimality condition excludes such situations. It should be noted that a foolproof version of
Meta-algorithm~\ref{alg:NearestPointRobust} must keep track of whether a correction of $\alpha_n(\varepsilon)$ on
Steps~3 and 4 has already been attempted and send an error message, if the method tries to correct the coefficients the
second time.
}
\end{remark}

\subsection{Numerical experiments}
\label{subsect:NumExperiments_NPP}

The proposed acceleration technique was verified numerically via multiple experiments with various values of $d$ 
and $\ell$. For each choice of $d$ and $\ell$ we randomly generated 10 problems and average computation time for 
these problems was used as a performance measure.

The problem data was chosen as follows. First, we randomly generated $\ell$ points 
$\{ \widehat{x}_1, \ldots, \widehat{x}_{\ell} \}$ in $\mathbb{R}^d$, uniformly distributed over the $d$-dimensional cube
$[-1, 1]^d$. Then these points were compressed and shifted as follows:
\[
  x_i = (1 + 0.01 \widehat{x}_i^{(1)}, \widehat{x}_i^{(2)}, \ldots, \widehat{x}_i^{(d)}) 
  \quad \forall i \in \{ 1, \ldots, \ell \}.
\] 
The point $z$ was chosen as $z = 0$. As was noted in \cite{Wolfe} and demonstrated by our numerical experiments, this
particular problem is very challenging for methods for finding the nearest point in a polytope (especially in the cases
when either $\ell$ is much greater than $d$ or $d$ is large) both in terms of computation time and finding an optimal
solution with high accuracy. For our numerical experiments we used the values $d = 3$, $d = 10$, and $d = 50$, while
values of $\ell$ were chosen from the set 
\[
  \{ 100, 200, 300, 500, 1000, 2000, 3000, 5000, 10000, 20000, 30000, 50000 \}
\]
and depended on $d$.

\begin{remark}
Let us note that we performed numerical experiments for many other values of $d$ and $\ell$, as well as, for many other
types of problem data. Since the results of corresponding numerical experiments were qualitatively the same as the ones
presented below, we do not include them here for the sake of shortness.
\end{remark}

Without trying to conduct exhaustive numerical experiments, we tested our acceleration technique on 3 classic methods:
the MDM method \cite{MDM}, the Wolfe method \cite{Wolfe}, and a method based on quadratic programming.
All methods were implemented in \textsc{Matlab}. The last method was based on solving the problem
\[
  \min_{\alpha} \frac{1}{2} \Big\| \sum_{i = 1}^{\ell} \alpha^{(i)} x_i \Big\|^2 \quad \text{subject to} \quad
  \sum_{i = 1}^{\ell} \alpha^{(i)} = 1, \quad \alpha^{(i)} \ge 0, \quad i \in I
\]
with the use of \texttt{quadprog}, the standard \textsc{Matlab} routine for solving quadratic programming problems. 
We used this routine with default settings. The number of iterations of the MDM and Wolfe method was limited to $10^6$. 
We used the inequalities
\[
  \delta(v_k) < \varepsilon, \quad \langle X, P_J \rangle > \langle X, X \rangle - \varepsilon
\]
with $\varepsilon = 10^{-4}$ as termination criteria for the MDM method and the Wolfe method respectively (see the
descriptions of these methods in \cite{MDM,Wolfe}). The value $10^{-4}$ was used, since occasionally both methods failed
to terminate with smaller value of $\varepsilon$ for large $\ell$ (this statement was especially true for the MDM
method).

Finally, we implemented each method ``on its own'' and also incorporated each method within the robust acceleration
technique, that is, Meta-algorithm~\ref{alg:NearestPointRobust}. The initial guess for the meta-algorithm was chosen as
\[
  I_0 = \{ 1, \ldots, d + 1 \}, \quad P_0 = \co\{ x_1, \ldots, x_{d + 1} \}.
\]
To demonstrate that the estimate of $\eta$ in Lemma~\ref{lem:ApproxDistanceDecay} and
Theorems~\ref{thrm:RobustMetaAlg_Wolfe} and \ref{thm:RobustMetaAlgCorrectness} (see \eqref{eq:ParametersConsistency}) is
very conservative, we used the value $\eta = 10^{-4}$ for $d = 3$ and $d = 10$, and  $\eta = 5 \cdot 10^{-4}$ for 
$d = 50$, since the acceleration technique occasionally failed to find a point satisfying the stopping criterion for
$\eta = 10^{-4}$ in this case. In addition, we terminated the computations, if computation time exceeded 1 minute.

\begin{figure}[t] 
\includegraphics[width=0.5\textwidth]{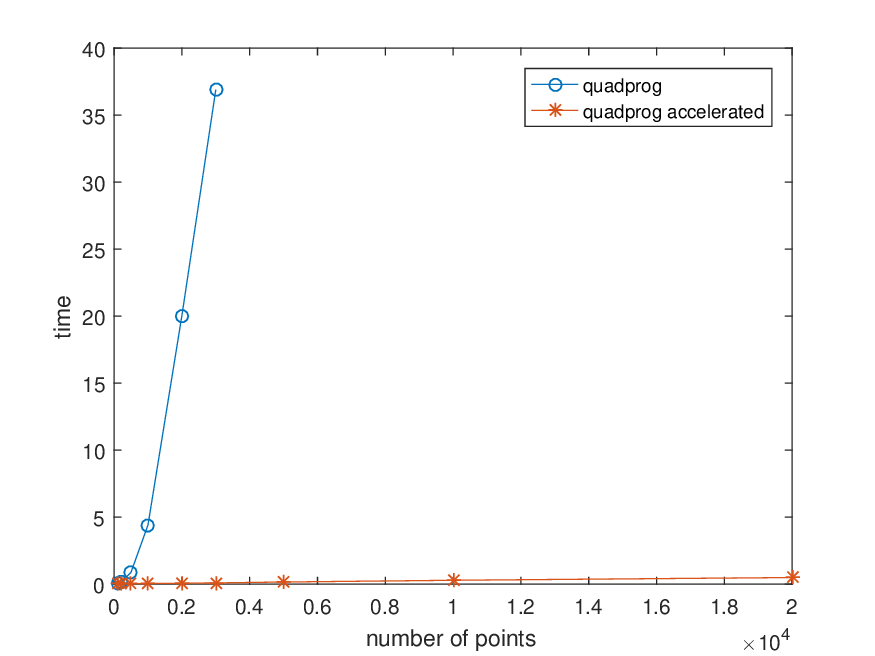}
\includegraphics[width=0.5\textwidth]{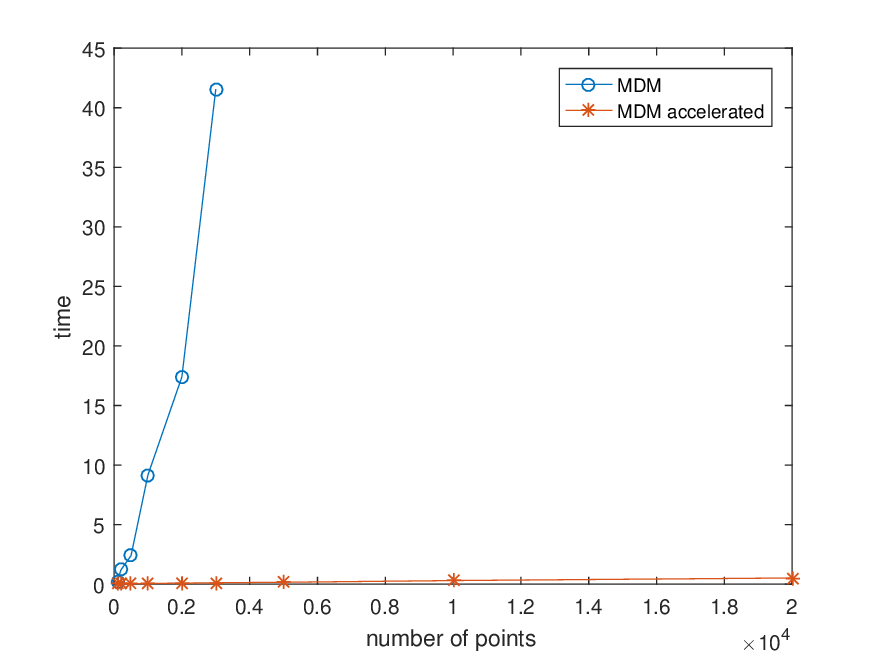}
\caption{The results of numerical experiments in the case $d = 3$ for \texttt{quadprog} routine (left figure) and
the MDM method (right figure).}
\label{fig:NearestPoint3}
\end{figure}

\begin{figure}[t]
\includegraphics[width=0.5\textwidth]{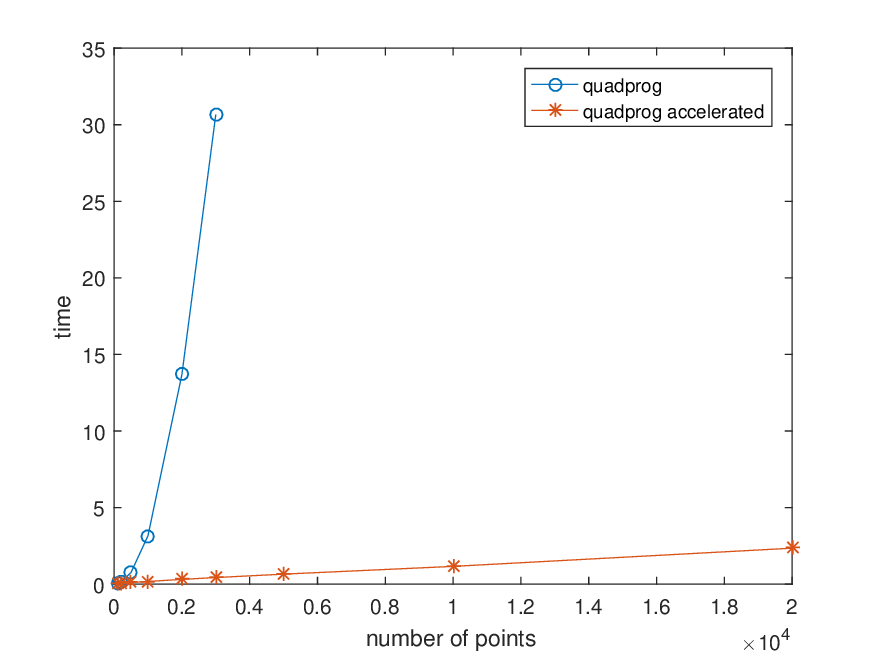}
\includegraphics[width=0.5\textwidth]{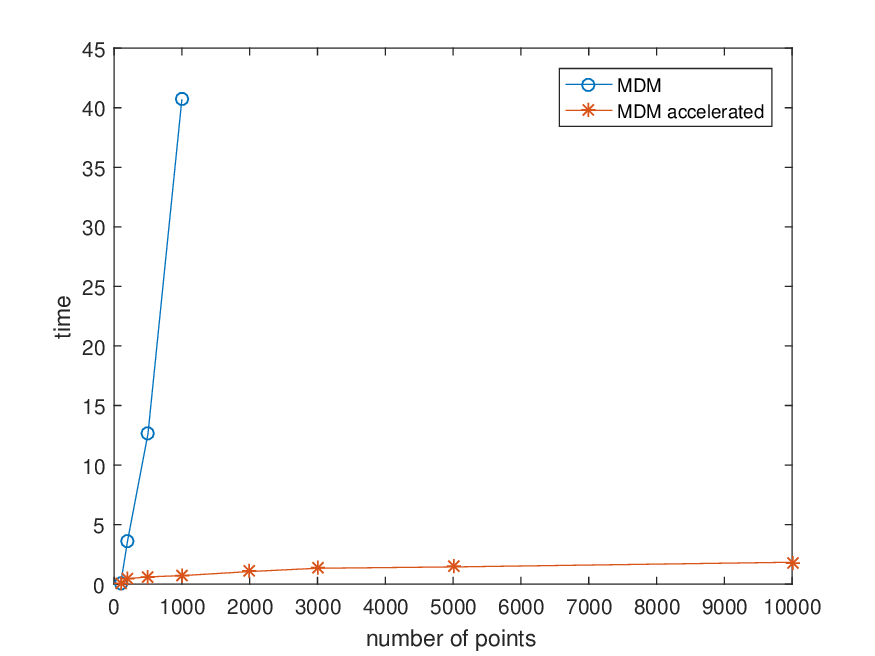}
\caption{The results of numerical experiments in the case $d = 10$ for \texttt{quadprog} routine (left figure) and
the MDM method (right figure).}
\label{fig:NearestPoint10}
\end{figure}

\begin{figure}[t] 
\includegraphics[width=0.5\textwidth]{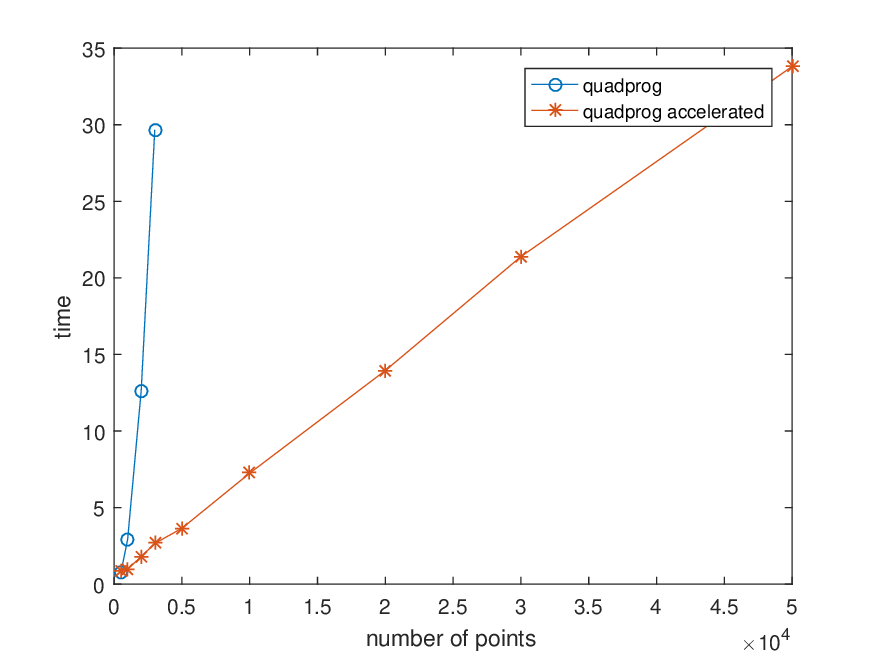}
\includegraphics[width=0.5\textwidth]{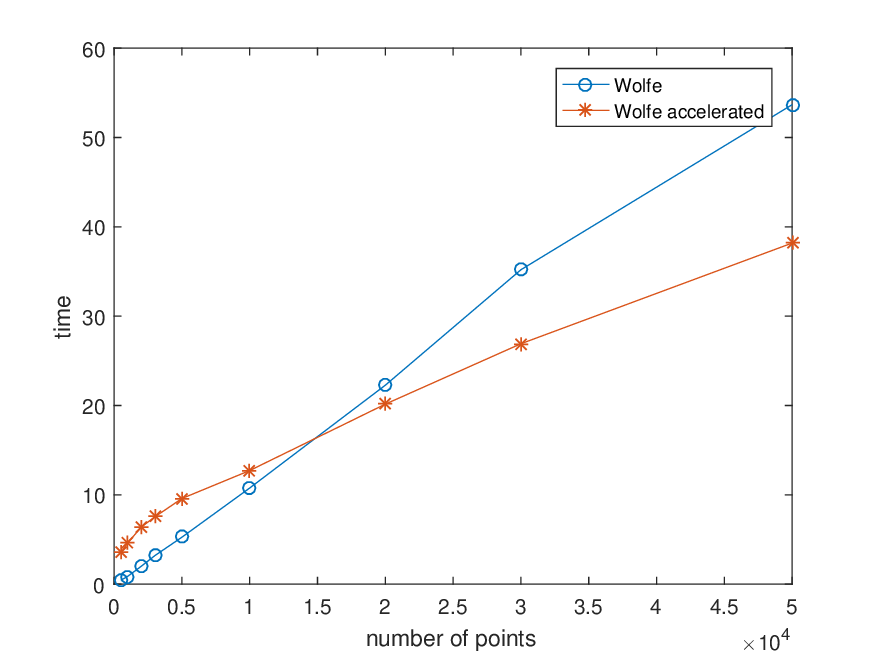}
\caption{The results of numerical experiments in the case $d = 50$ for \texttt{quadprog} routine (left figure) and
the Wolfe method (right figure).}
\label{fig:NearestPoint50}
\end{figure}

The results of numerical experiments are given in Figs.~\ref{fig:NearestPoint3}--\ref{fig:NearestPoint50}. Let us first 
note that both the MDM method and its accelerated version were very slow and inaccurate in comparison with other methods
in the case $d = 50$. That is why we do not present the corresponding results of the numerical experiments here. 
Furthermore, our numerical experiments showed that the difference in performance between the Wolfe method and its 
accelerated version significantly increases with the growth of $d$. Since the results were qualitatively the same 
for all $d$, here we present them only in the case $d = 50$, in which the difference in performance was the most
noticeable.

The numerical experiments clearly demonstrate that the proposed acceleration technique with the steepest descent
exchange rule allows one to significantly reduce the computation time for methods of finding the nearest point in a
polytope. In the case of \texttt{quadprog} routine and the MDM method the reduction in time is proportional to the
number of points $\ell$. Moreover, numerical experiments also showed that for the accelerated versions of these
methods the computation time increases \textit{linearly} in $\ell$ for the problem under consideration (but we do not
claim the linear in $\ell$ complexity of the acceleration technique for all types of problem data). 

In the case of the Wolfe method the situation is somewhat different. For relatively small values of $\ell$ the ``pure''
Wolfe method outperforms its accelerated version, but for larger $\ell$ the accelerated version is faster than the pure
method and the reduction of computation time increases with the increase of $\ell$. However, it should be noted that
precisely the same effect was observed for all other methods for finding the nearest point in a polytope and other
types of problem data. When $\ell$ is close to $d$, it is faster to solve the corresponding problem with the use of 
the algorithm $\mathcal{A}_{\varepsilon}$, while when $\ell$ exceeds a certain threshold (depending on a particular
method), the accelerated version of the algorithm $\mathcal{A}_{\varepsilon}$ starts to outperform the algorithm 
$\mathcal{A}_{\varepsilon}$ ``on its own''. Moreover, the reduction of the run-time increases with an increase
of $\ell$ and is proportional to $\ell$ for large values of $\ell$. 

Thus, the main difference between the Wolfe method and other methods tested in our numerical experiments lies in 
the fact that the threshold value of $\ell$, after which the acceleration technique becomes efficient, is significantly
higher for the Wolfe method that for other methods.

Finally, let us note that the average number of iterations of Meta-algorithm~\ref{alg:NearestPointRobust}, as expected,
depended on the method to which this technique was applied. It was the highest for the MDM method, while the number of
iterations of the meta-algorithm using the Wolfe method and \texttt{quadprog} routine was roughly the same. In the case
$d = 3$ it was equal to 6, in the case $d = 10$ it was equal to 25.6, while in the case $d = 50$ it was equal to 150.8
(note that the problem is harder to solve for larger $d$). The results of other multiple numerical experiments, not
reported here for the sake of shortness, showed that the average number of iterations of the meta-algorithm in most
cases lies between $d$ and $10d$ for various values of $d$ and $\ell$. Since the complexity of each iteration of the
method is proportional to $\ell$, the results of numerical experiments hint at the linear in $\ell$ average case
complexity of Meta-algorithm~\ref{alg:NearestPointRobust}, but we do not claim this complexity estimate to be true in
the general case (especially, in the worst case) and are not ready to provide its theoretical justification.

It should be noted that an increase of average number of iterations with an increase of $d$ is fully expected. 
To understand a reason behind it, observe that if the projection of $z$ onto the polytope 
$P = \co\{ x_1, \ldots, x_{\ell} \}$ belongs to the relative interior of a facet $F$ of $P$, then to find this
projection the meta-algorithm needs to find a subpolytope $P_n$ containing at least $d$ extreme point of $F$. If none of
these points belongs to the initial guess $P_0 = \co\{ x_1, \ldots, x_{d + 1} \}$, then at at least $d$ iterations are
needed to find the required polytope $P_n$. 

For example, if in the case $d = 2$ the polytope $P_0$ lies in the interior of $P$, then one needs at least 2
iterations for the polytope $P_n$ to contain the edge of $P$ to which the projection of $z$ onto $P$ belongs. In the
case $d = 3$, the minimal number of iterations increases to $3$, etc. Thus, the number of iterations of the
meta-algorithm grows whenever $d$ is increased. 

\begin{remark}
In our implementation of Meta-algo\-rithm~\ref{alg:NearestPointRobust}, the algorithm $\mathcal{A}_{\varepsilon}$ was
applied afresh on each iteration (i.e. without using any information from the previous iteration). It should be noted
that, in particular, the performance of the accelerated version of the Wolfe method can be significantly improved, if
one uses the \textit{corral} (see \cite{Wolfe}) computed on the previous iteration as the initial guess for the next
iteration (a similar remark is true for accelerated versions of other methods). However, since our main goal was to
demonstrate the performance of the acceleration technique on its own, here we do not discuss potential ways this
technique can be efficiently integrated with a particular method for finding the nearest point in a polytope and do not
present any results of numerical experiments for such fully integrated accelerated methods.
\end{remark}

\section{A comparison with the Wolfe method}
\label{sect:ComparisonWithWolfe}

Meta-algorithm \ref{alg:NearestPoint} with the steepest descent exchange rule and
Meta-algorithm~\ref{alg:NearestPointRobust} share many similarities with the Wolfe method \cite{Wolfe} (and the
Frank-Wolfe algorithms \cite{LacosteJulienJaggi2013,LacosteJulienJaggi2015}). Nonetheless, there is one important
difference between these meta-algorithms and the Wolfe method, which, as the results of numerical
experiments presented in the previous section demonstrate (see Fig.~\ref{fig:NearestPoint50}), allows 
Meta-algorithm~\ref{alg:NearestPointRobust} to outperform the Wolfe method \cite{Wolfe} in the case when the number of
points is significantly greater than the dimension of the space.

The difference consists in the way in which the steepest descent exchange rule and the Wolfe method remove redundant
points on each iteration. The Wolfe method operates with the so-called \textit{corrals} (i.e. an affinely independent
set of points from the polytope such that the projection of the origin onto the affine hull of this set belongs to the
relative interior of its convex hull), while Meta-algorithm \ref{alg:NearestPoint} with the steepest descent exchange
rule and Meta-algorithm~\ref{alg:NearestPointRobust} operate with convex hulls of $d + 1$ points without imposing any
assumptions on them. On each iteration of the Wolfe method, given a current corral $Q_n$, one adds a new point $x_n$ to
this corral in the same way points are added in the steepest descent exchange rule, and then constructs a new corral
$Q_{n + 1}$ from the set $\{ x_n, Q_n \}$, filtering out multiple redundant points in the general case. In contrast,
Meta-algorithm \ref{alg:NearestPoint} with the steepest descent exchange rule and
Meta-algorithm~\ref{alg:NearestPointRobust} remove only \textit{one} point on each iteration.

This difference, apart from saving significant amount of time needed to find a corral in the spaces of moderate and
large dimensions, also leads to a significantly different behaviour of Meta-algorithms \ref{alg:NearestPoint} and
\ref{alg:NearestPointRobust} in comparison with the Wolfe method in the general case. This difference in behaviour
occurs due to the fact that the projection of a given point onto  the ``unfiltered'' subpolytope used by the
meta-algorithms might be significantly different from the projection of a point onto the corral constructed by the Wolfe
method. The following simple example highlights this difference.

\begin{example}
Let $d = 2$, $\ell = 4$, $z = 0$, and
\[
  x_1 = (0, 4), \quad x_2 = (0, 2), \quad x_3 = (2, 2), \quad x_4 = (-2, 1).
\]
First, we consider the behaviour of the Wolfe method. We use the same notation as in Wolfe's original 
paper \cite{Wolfe}.
\begin{itemize}
\item{\textit{Step 0:} The point with minimal norm is $x_2$. Therefore define $S = \{ 2 \}$ and $w = 1$.}

\item{\textit{Iteration (major cycle) 1:} 
  \begin{itemize}
    \item{\textit{Step 1:} The point $X = x_2$ does not satisfy the stopping criterion. The minimum in
    $\min_J \langle X, x_J \rangle$ is attained for $J = 4$. Therefore, put $S = \{ 2, 4 \}$ and $w = (1, 0)$.}

    \item{\textit{Step 2:} Solving \cite[Eq.~(4.1)]{Wolfe} one gets $v = (0.6, 0.4)$. Put $w = v$.}
  \end{itemize}
}

\item{\textit{Iteration (major cycle) 2:} 
  \begin{itemize}
    \item{\textit{Step 1:} The point $X = 0.6 x_2 + 0.4 x_4 = (-0.8, 1.6)$ does not satisfy the stopping criterion. 
    The minimum in $\min_J \langle X, x_J \rangle$ is attained for $J = 3$. Therefore, put $S = \{ 2, 4, 3 \}$ and 
    $w = (0.6, 0.4, 0)$.}

    \item{\textit{Step 2:} Solving \cite[Eq.~(4.1)]{Wolfe} one gets $v = (0, 10/17, 7/17)$.}

    \item{\textit{Step 3:} One has POS $= \{ 2 \}$, $\theta = 1$, $w = v$. Therefore, remove point $x_2$, and
    put $S = \{ 4, 3 \}$ and $w = (10/17, 7/17)$.}

    \item{\textit{Step 2:} Solving \cite[Eq.~(4.1)]{Wolfe} one gets $v = (10/17, 7/17)$. Put $w = v$.}
  \end{itemize}
}

\item{\textit{Iteration (major cycle) 3: 
  \begin{itemize}
    \item{\textit{Step 1:} The point $X = (10/17) x_4 + (7/17) x_3 = (-6/17, 24/17)$ satisfies the stopping criterion.}
  \end{itemize}
  }
}
\end{itemize}
Let us now consider the behaviour of Meta-algorithm \ref{alg:NearestPoint} with the steepest descent exchange rule.
\begin{itemize}
\item{\textit{Initialization:} Let $I_0 = \{ 1, 2, 3 \}$. Put $P_0 = \co\{ x_1, x_2, x_3 \}$.
}

\item{\textit{Iteration 0:}
  \begin{itemize}
    \item{\textit{Step 1:} One has $\alpha_n = \mathcal{A}(z, P_0) = (0, 1, 0)$ and $y_n = x_2$. The optimality
    condition is not satisfied.}
  
    \item{\textit{Step 2:} By employing the steepest descent exchange rule, one gets $i_{10} = 1$ and $i_{20} = 4$.
    Therefore, $I_1 = \{ 2, 3, 4 \}$ and $P_1 = \co\{ x_2, x_3, x_4 \}$.}
  \end{itemize}
}

\item{\textit{Iteration 1:}
  \begin{itemize}
    \item{\textit{Step 1:} One has $\alpha_n = \mathcal{A}(z, P_1) = (0, 7/17, 10/17)$ and $y_n = (-6/17, 24/17)$. The
    optimality condition is satisfied and the meta-algorithm terminates.}
  \end{itemize}
}
\end{itemize}
Note that by not removing the ``redundant'' point $x_3$, Meta-algorithm \ref{alg:NearestPoint} is able to find the
optimal solution in just one iteration, while the Wolfe method, operating with corrals, needs to do several iterations
to find the corral $\{ x_3, x_4 \}$ containing the required projection.
\end{example}

\begin{figure}[t] 
\includegraphics[width=0.75\textwidth]{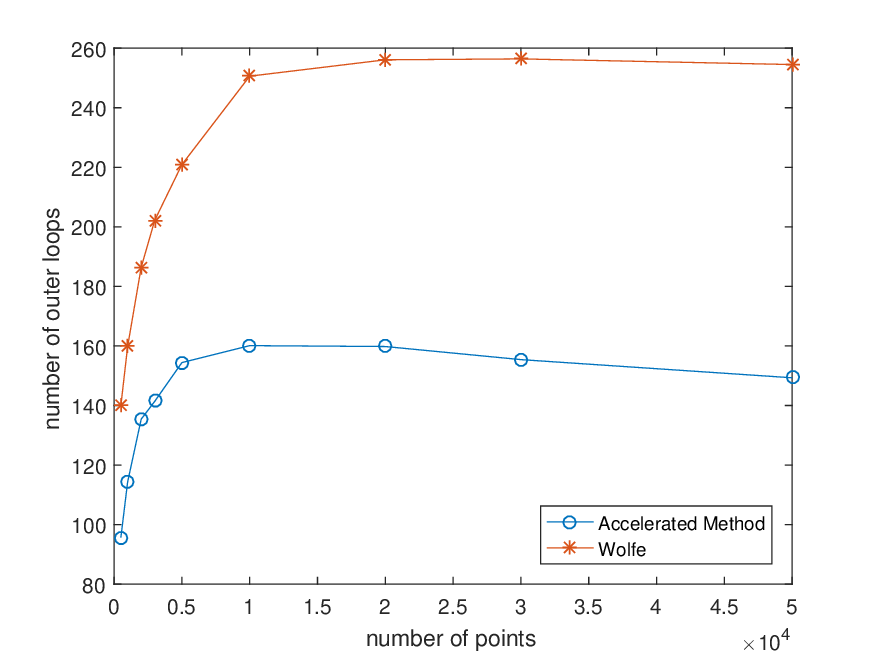}
\caption{The average number of outer loops (iterations/shifts of the subpolytope) for Meta-algorithm
\ref{alg:NearestPointRobust} and the Wolfe method in the case $d = 50$.}
\label{fig:NumberOfOuterLoops}
\end{figure}

In order to futher demonstrate how the different strategies of removing redundant points affect the performance of both
methods, we present the average number of outer loops (iterations/shifts of the subpolytope) for
Meta-algorithm~\ref{alg:NearestPointRobust} and the Wolfe method in the case $d = 50$ (see
Figure~\ref{fig:NumberOfOuterLoops}). Observe that the number of iterations of the meta-algorithm is significantly
smaller than the number of iterations of the Wolfe method, which explains why for large $\ell$ the meta-algorithm
outperformes the Wolfe method in terms of computation time. 

For the given problem data, the average number of iterations of the meta-algorithm stabilizes around 150 iterations for
large $\ell$, while in the case of the Wolfe method it stabilizes around 250 iterations for large $\ell$. Thus, 
the completely different strategy of removing redundant points allows the meta-algorithm to reduce the number of
iterations by about $40\%$ in comparison with the Wolfe method.

\section{Computing the distance between two polytopes}
\label{sect:DistanceBetweenPolytopes}

The acceleration technique for methods for finding the nearest point in a polytope from the previous section
admits a straightforward extension to the case of methods for computing the distance between two convex polytopes
defined as the convex hulls of finite sets of points. This section is devoted to a detailed discussion of such
extension, that is, to a discussion of an acceleration technique for methods of solving the following
optimization problem
\[
  \min_{(x, y)} \: \| x - y \| \enspace \text{ s.t. } \enspace
  x \in P := \co\big\{ x_1, \ldots, x_{\ell} \big\}, \enspace
  y \in Q = \co\big\{ y_1, \ldots, y_m \}.
  \quad \eqno{(\mathcal{D})}
\]
in the case when $\ell \gg d$ and $m \gg d$. Here $x_i \in \mathbb{R}^d$, $i \in I = \{ 1, \ldots, \ell \}$, and
$y_j \in \mathbb{R}^d$, $j \in J = \{ 1, \ldots, m \}$, are given points.

\subsection{An extension of the acceleration technique}
\label{subsect:TheoreticalDistance}

Before we proceed to the description of an acceleration technique, let us present convenient optimality conditions
for the problem $(\mathcal{D})$. These conditions are well-known, but we include their short proofs for the sake of
completeness. Denote by $Pr_A(x)$ the Euclidean projection of a point $x \in \mathbb{R}^d$ on a closed convex set 
$A \subset \mathbb{R}^d$, i.e. an optimal solution of the problem $\min_{y \in A} \| x - y \|$.

\begin{proposition} \label{prp:PolytopesDistOptCond}
A pair $(x_*, y_*) \in P \times Q$ is an optimal solution of the problem $(\mathcal{D})$ if and only if
$Pr_P(y_*) = x_*$ and $Pr_Q(x_*) = y_*$ or, equivalently,
\begin{equation} \label{eq:SeparateOptimality}
  \langle x_* - y_*, x_i - x_* \rangle \ge 0 \quad \forall i \in I, \quad
  \langle y_* - x_*, y_j - y_* \rangle \ge 0 \quad \forall j \in J.
\end{equation}
\end{proposition}

\begin{proof}
The fact that conditions \eqref{eq:SeparateOptimality} are equivalent to the equalities $Pr_P(y_*) = x_*$ and 
$Pr_Q(x_*) = y_*$ follows directly from Prop.~\ref{prp:NearPoint_OptCond}. Let us check that these conditions are
equivalent 
to the optimality of $(x_*, y_*)$.

Indeed, as is easily seen, inequalities \eqref{eq:SeparateOptimality} are satisfied if and only if
\[
  \langle x_* - y_*, x - x_* \rangle \ge 0 \quad \forall x \in P, \quad
  \langle y_* - x_*, y - y_* \rangle \ge 0 \quad \forall y \in Q.
\]
In turn, these conditions are satisfied if and only if
\[
  \langle x_* - y_*, x - x_* \rangle + \langle y_* - x_*, y - y_* \rangle \ge 0
  \quad \forall x \in P, \: y \in Q. 
\] 
It remains to note that by standard optimality conditions for convex programming problems the above inequality
is equivalent to the optimality of $(x_*, y_*)$.
\end{proof}

Suppose that an algorithm $\mathcal{A}$ for solving the problem $(\mathcal{D})$ is given. For any two polytopes
$V, W \subset \mathbb{R}^d$, defined as the convex hulls of finite collections of points, it returns an optimal solution
$(v_*, w_*) = \mathcal{A}(V, W)$ of the problem
\[
  \min_{(v, w)} \: \| v - w \| \quad \text{subject to} \quad v \in V, \quad w \in W.
\]
We propose to accelerate this algorithm in precisely the same way as an algorithm for solving the nearest point problem.
Namely, one chooses ``small'' subpolytopes $P_0 \subset P$ and $Q_0 \subset Q$ and applies the algorithm $\mathcal{A}$
to find the distance between these subpolytopes. If an optimal solution of this problem coincides with an optimal
solution of the problem $(\mathcal{D})$ (this fact is verified via the optimality conditions from
Prop.~\ref{prp:PolytopesDistOptCond}), then the computations are terminated. Otherwise, one shifts these polytopes and 
repeats the same procedure till an optimal solution of the distance problem for subpolytopes $P_n$ and $Q_n$ coincides 
with an optimal solution of the problem $(\mathcal{D})$. A general structure of this acceleration technique
(meta-algorithm) is essentially the same as the structure of Meta-algorithm~\ref{alg:NearestPoint} and is given in 
Meta-algorithm~\ref{alg:PolytopesDistance}. For the sake of simplicity we suppose that the same exchange rule is used
for each polytope, and the subpolytopes $P_n$ and $Q_n$ are convex hulls of the same number of points.

\begin{algorithm}[t]  \label{alg:PolytopesDistance}
\caption{Meta-algorithm for finding the distance between two polytopes.}

\noindent\textbf{Input:} {two collection of points $\{ x_1, \ldots, x_{\ell} \} \subset \mathbb{R}^d$ and 
$\{ y_1, \ldots, y_m \} \subset \mathbb{R}^d$, an algorithm $\mathcal{A}$ for solving the problem of computing
the distance between two polytopes, parameters $s, q \in \{ 1, \ldots, \min\{ \ell, m \} \}$ with $s \ge q$, and 
an $(s, q)$-exchange rule $\mathcal{E}$.}

\noindent\textbf{Initialization:} {Put $n = 0$, choose index sets $I_n \subseteq I$ and $J_n \subseteq J$ with 
$|I_n| = |J_n| = s$. Define $P_n = \co\{ x_i \mid i \in I_n \}$ and $Q_n = \co\{ y_j \mid j \in J_n \}$.}

\noindent\textbf{Step 1.} {Compute $(v_n, w_n) = \mathcal{A}(P_n, Q_n)$ and
\[
  \rho_{xn} = \min_{i \in I} \langle v_n - w_n, x_i - v_n \rangle, \quad
  \rho_{yn} = \min_{j \in J} \langle w_n - v_n, y_j - w_n \rangle.
\]
If $\rho_{xn} \ge 0$ and $\rho_{yn} \ge 0$, \textbf{return} $(v_n, w_n)$.}

\noindent\textbf{Step 2.} {If $\rho_{xn} < 0$, compute $(I_{1n}, I_{2n}) = \mathcal{E}(I_n)$ and define
\[
  I_{n + 1} = \Big( I_n \setminus I_{1n} \Big) \cup I_{2n}, \quad 
  P_{n + 1} = \co\big\{ x_i \bigm| i \in I_{n + 1} \big\}.
\]
Otherwise, set $I_{n + 1} = I_n$ and $P_{n + 1} = P_n$.

If $\rho_{yn} < 0$, compute $(J_{1n}, J_{2n}) = \mathcal{E}(J_n)$ and define
\[
  J_{n + 1} = \Big( J_n \setminus J_{1n} \Big) \cup J_{2n}, \quad 
  Q_{n + 1} = \co\big\{ y_j \bigm| j \in J_{n + 1} \big\}.
\]
Otherwise, define $J_{n + 1} = J_n$ and $Q_{n + 1} = Q_n$. Put $n = n + 1$ and go to \textbf{Step 1}.}
\end{algorithm}

Arguing in essentially the same way as in the proofs of Lemmas~\ref{lem:DistDecay} and \ref{lem:SubpolytopeSize} one 
can verify that the following results hold true. These results can be viewed as criteria for choosing an efficient
exchange rule for Meta-algorithm~\ref{alg:PolytopesDistance}.

For any sets $A, B \subset \mathbb{R}^d$ denote by
\[
  \dist(A, B) = \inf\big\{ \| x - y \| \bigm| x \in A, \: y \in B \big\}
\]
the Euclidean distance between these sets.

\begin{lemma}
Suppose that the exchange rule $\mathcal{E}$ satisfies the distance decay condition for the problem $(\mathcal{D})$: 
if for some $n \in \mathbb{N}$ the pair $(v_n, w_n)$ does not satisfy the optimality conditions $\rho_{xn} \ge 0$ and 
$\rho_{yn} \ge 0$, then
\begin{equation} \label{eq:DistDecayCondition}
  \dist(P_{n + 1}, Q_{n + 1}) < \dist(P_n, Q_n).
\end{equation}
Then Meta-algorithm~\ref{alg:PolytopesDistance} terminates after a finite number of steps and returns an optimal
solution of the problem $(\mathcal{D})$.
\end{lemma}

\begin{lemma}
Let an $(s, q)$-exchange rule $\mathcal{E}$ with $s \ge q$ satisfy the distance decay condition for the problem
$(\mathcal{D})$ for any polytopes $U, W \subset \mathbb{R}^d$. Then $s \ge d + 1$.
\end{lemma}

As in the case of the nearest point problem, one can utilise the steepest descent exchange rule to shift polytopes $P_n$
and $Q_n$ on each iteration of Meta-algorithm~\ref{alg:PolytopesDistance}. In the case of the problem $(\mathcal{D})$
this exchange rule is defined as follows (for the sake of shortness we describe it only for the polytope
$P_n$).

\begin{itemize}
\item{\textbf{Input:} an index set $I_n \subset I$ with $|I_n| = d + 1$, the set $\{ x_1, \ldots, x_{\ell} \}$, and
the pair $(v_n, w_n) = \mathcal{A}(P_n, Q_n)$.
}

\item{\textbf{Step 1:} Find $i_{1n} \in I_n$ such that $v_n \in \co\{ x_i \mid I_n \setminus \{ i_{1n} \} \}$.}

\item{\textbf{Step 2:} Find $i_{2n} \in I$ such that
\[
  \langle v_n - w_n, x_{i_{2n}} \rangle = \min_{i \in I} \langle v_n - w_n, x_i \rangle.
\]
\textbf{Return} $(\{ i_{1n} \}, \{ i_{2n} \})$.}
\end{itemize}

The following theorem, whose prove is similar to the proof of Theorem~\ref{thm:StDescExchRule}, shows that the steepest
descent exchange rule satisfies the distance decay condition for the problem $(\mathcal{D})$.

\begin{theorem}
For any polytopes $P = \co\{ x_1, \ldots, x_{\ell} \}$ and $Q = co\{ y_1, \ldots, y_m \}$ with $\ell \ge d + 1$ and
$m \ge d + 1$ the steepest descent exchange rule satisfies the distance decay condition for the problem
$(\mathcal{D})$.
\end{theorem}

\begin{proof}
Suppose that for some $n \in \mathbb{N}$ the pair $(v_n, w_n)$ does not satisfy the stopping criterion 
$\rho_{xn} \ge 0$ and $\rho_{yn} \ge 0$ from Meta-algorithm~\ref{alg:PolytopesDistance}. Let us consider three cases.

\textbf{Case I.} Let $\rho_{xn} < 0$ and $\rho_{yn} < 0$. Denote by 
\[
  P_n^0 = \co\Big\{ x_i \Bigm| i \in I_n \setminus \{ i_{1n} \} \Big\}, \quad
  Q_n^0 = \co\Big\{ y_j \Bigm| j \in J_n \setminus \{ j_{1n} \} \Big\}
\]
the polytopes obtained from $P_n$ and $Q_n$ after removing the points selected by the steepest descent exchange rule.
By the definition of this rule $v_n \in P_n^0$ and $w_n \in Q_n^0$ and, moreover,
\begin{equation} \label{eq:DoubleGradAtNewPoint}
  \langle v_n - w_n, x_{i_{2n}} - v_n \rangle = \rho_{xn} < 0, \quad
  \langle w_n - v_n, y_{i_{2n}} - w_n \rangle = \rho_{yn} < 0.
\end{equation}
Recall also that $P_{n + 1} = \co\{ P_n^0, x_{i_{2n}} \}$ and $Q_{n + 1} = \co\{ Q_n^0, y_{j_{2n}} \}$.

Introduce the vectors
\[
  x_n(t) = (1 - t) v_n + t x_{i_{2n}}, \quad y_n(\tau) = (1 - \tau) w_n + \tau y_{j_{2n}}
\]
and the function $f(t, \tau) = \| x_n(t) - y_n(\tau) \|^2$. Clearly, $x_n(t) \in P_{n + 1}$ for any $t \in [0, 1]$
and $y_n(\tau) \in Q_{n + 1}$ for any $\tau \in [0, 1]$. In addition, one has
\begin{align*}
  f(t, \tau) &= \| v_n - w_n + t (x_{i_{2n}} - v_n) - \tau (y_{j_{2n}} - w_n) \|^2
  \\
  &= \| v_n - w_n \|^2 + 2 t \rho_{xn} + 2 \tau \rho_{yn} - 2 t \tau \langle x_{i_{2n}} - v_n, y_{j_{2n}} - w_n \rangle
  \\
  &+ t^2 \| x_{i_{2n}} - v_n \| + \tau^2 \| y_{j_{2n}} - w_n \|^2.
\end{align*}
From inequalities \eqref{eq:DoubleGradAtNewPoint} it follows that for any sufficiently small $t \in (0, 1)$ and 
$\tau \in (0, 1)$ one has $f(t, \tau) < f(0, 0)$. Therefore for any such $t$ and $\tau$ one has
\begin{align*}
  \dist(P_{n + 1}, Q_{n + 1}) &\le \| x_n(t) - y_n(\tau) \| = \sqrt{f(t, \tau)} 
  \\
  &< \sqrt{f(0, 0)} = \| v_n - w_n \| = \dist(P_n, Q_n),
\end{align*}
which means that the distance decay condition for the problem $(\mathcal{D})$ holds true.

\textbf{Case II.} Suppose that $\rho_{xn} < 0$, but $\rho_{yn} \ge 0$. By definition $(v_n, w_n) \in P_n \times Q_n$ 
is an optimal solution of the problem
\[
  \min_{(x, y)} \| x - y \| \quad \text{subject to} \quad x \in P_n, \quad y \in Q_n.
\]
Hence by Proposition~\ref{prp:PolytopesDistOptCond} one has $Pr_{P_n}(w_n) = v_n$. Moreover, by our assumption one has
\[
  \min_{i \in I} \langle v_n - w_n, x_i - v_n \rangle < 0,
\]
that is, $v_n$ does not satisfy the optimality condition for the nearest point problem $\min_{x \in P_n} \| x - w_n \|$.
Therefore, almost literally repeating the proof of Theorem~\ref{thm:StDescExchRule} one gets that
\[
  \dist(w_n, P_{n + 1}) < \dist(w_n, P_n).
\]
Hence bearing in mind the facts that $Q_{n + 1} = Q_n$, $w_n \in Q_n$, and $\dist(w_n, P_n) = \dist(P_n, Q_n)$ one
obtains
\[
  \dist(P_{n + 1}, Q_{n + 1}) \le \dist(w_n, P_{n + 1}) < \dist(w_n, P_n) = \dist(P_n, Q_n).
\]
Thus, the distance decay condition for the problem $(\mathcal{D})$ holds true.

\textbf{Case III.} Suppose that $\rho_{yn} < 0$, but $\rho_{xn} \ge 0$. The proof of this case almost literally
repeats the proof of Case II. 
\end{proof}

\begin{corollary}
Let $\ell, m \ge d + 1$. Then Meta-algorithm~\ref{alg:PolytopesDistance} with $s = d + 1$, $q = 1$, and the steepest
descent exchange rule terminates after a finite number of iterations and returns an optimal solution of the problem
$(\mathcal{D})$.
\end{corollary}

As in the case of the nearest point problem, it is easier to implement the steepest descent exchange rule for 
the problem $(\mathcal{D})$, when the algorithm $\mathcal{A}$ returns not an optimal solution 
$(v_*, w_*) = \mathcal{A}(V, W)$ of the problem
\[
  \min \: \| v - w \| \quad \text{subject to} \quad 
  V = \co\{ v_1, \ldots, v_{\ell_1} \}, \quad W = \co\{ w_1, \ldots, w_{\ell_2} \},
\]
but coefficients of the corresponding convex combinations, that is, vectors $\alpha \in \mathbb{R}^{\ell_1}$
and $\beta \in \mathbb{R}^{\ell_2}$ such that
\begin{align*}
  v_* &= \sum_{i = 1}^{\ell_1} \alpha^{(i)} v_i, \quad \sum_{i = 1}^{\ell_1} \alpha^{(i)} = 1, \quad
  \alpha^{(i)} \ge 0 \quad \forall i \in \{ 1, \ldots, \ell_1 \},
  \\
  w_* &= \sum_{j = 1}^{\ell_2} \beta^{(j)} w_j, \quad \sum_{j = 1}^{\ell_2} \beta^{(j)} = 1, \quad
  \beta^{(j)} \ge 0 \quad \forall j \in \{ 1, \ldots, \ell_2 \}.
\end{align*}
In this case, if on iteration $n$ of Meta-algorithm~\ref{alg:PolytopesDistance} one computes a pair of coefficients
of convex combinations $(\alpha_n, \beta_n) = \mathcal{A}(P_n, Q_n)$, then one can obviously choose as an index 
$i_{1n} \in I_n$ of vector $x_{i_{1n}}$ that is removed from $P_n$ any index $k \in I_n$ such that 
$\alpha_n^{(k)} = 0$. Similarly, one can choose as an index $j_{1n}$ of a vector $y_{j_{1n}}$ that is removed from $Q_n$
any index $k \in J_n$ such that $\beta_n^{(k)} = 0$.

If the polytopes $P_n$ and $Q_n$ intersect, then Meta-algorithm~\ref{alg:PolytopesDistance} terminates on iteration
$n$. If they do not intersect, then the points $v_n$ and $w_n$ obviously lie on the boundaries of $P_n$ and $Q_n$
respectively. Hence by \cite[Lemma~2.8]{Ziegler} in the case when the points $x_i$, $i \in I_n$, are affinely
independent, there exists at least one $k \in I_n$ such that $\alpha_n^{(k)} = 0$. Similarly, in the case when the
points $y_j$, $j \in J_n$, are affinely independent, there exists at least one $k \in J_n$ such that 
$\beta_n^{(k)} = 0$. In this case one can easily find the required indices $i_{1n} \in I_n$ and $j_{1n} \in J_n$.

If either $x_i$, $i \in I_n$, or $y_j$, $j \in J_n$, are affinely dependent, then one can utilise the following obvious
extension of \textit{the index removal method} from Section~\ref{subsect:ExchangeRule}. For the sake of shortness we
describe in only in the case when $\rho_{xn} < 0$ and $\rho_{yn} < 0$.

\begin{itemize}
\item{\textbf{Input:} index sets $I_n \subset I$, $J_n \subset J$ with $|I_n| = |J_n| = d + 1$, the sets 
$\{ x_1, \ldots, x_{\ell} \}$ and $\{ y_1, \ldots, y_m \}$, and $(\alpha_n, \beta_n) = \mathcal{A}(P_n, Q_n)$.
}

\item{\textbf{Step 1:} Compute $\alpha_{\min} = \min_{i \in I_n} \alpha_n^{(i)}$. If $\alpha_{\min} = 0$,
find $i_{1n} \in I_n$ such that $\alpha_n^{(i_{1n})} = 0$. Otherwise, choose any $k \in I_n$ and compute a least-squares
solution $\gamma_n$ of the system
\[
  \sum_{i \in I_n \setminus \{ k \}} \gamma^{(i)} (x_i - x_k) = 0, \quad 
  \sum_{i \in I_n \setminus \{ k \}} \gamma^{(i)} = 1
\]
and set $\gamma_n^{(k)} = - 1$. Find an index $i_{1n} \in I_n$ on which the minimum in
\[
  \min\left\{ - \frac{\alpha_n^{(i)}}{\gamma_n^{(i)}} \Biggm| 
  i \in I_n \colon \gamma_n^{(i)} < 0 \right\}
\]
is attained.
}

\item{\textbf{Step 2:} Compute $\beta_{\min} = \min_{j \in J_n} \beta_n^{(i)}$. If $\beta_{\min} = 0$,
find $j_{1n} \in J_n$ such that $\beta_n^{(j_{1n})} = 0$. Otherwise, choose any $k \in J_n$ and compute a least-squares
solution $\lambda_n$ of the system
\[
  \sum_{j \in J_n \setminus \{ k \}} \lambda^{(j)} (y_j - y_k) = 0, \quad 
  \sum_{j \in J_n \setminus \{ k \}} \lambda^{(j)} = 1
\]
and set $\lambda_n^{(k)} = - 1$. Find an index $j_{1n} \in J_n$ on which the minimum in
\[
  \min\left\{ - \frac{\beta_n^{(j)}}{\lambda_n^{(j)}} \Biggm| 
  j \in J_n \colon \lambda_n^{(j)} < 0 \right\}
\]
is attained and \textbf{return} $(i_{1n}, j_{1n})$.
}
\end{itemize}

Arguing in precisely the same way as in the proof of Proposition~\ref{prp:IndexRemovalMethod} one can verify that the
index removal method correctly finds the required indices $i_{1n} \in I_n$ and $j_{1n} \in J_n$.

\begin{proposition}
Suppose that for some $n \in \mathbb{N}$ the stopping criterion $\rho_{xn} \ge 0$ and $\rho_{yn} \ge 0$ of
Meta-algorithm~\ref{alg:PolytopesDistance} does not hold true, and let $(i_{1n}, j_{1n}) \in I_n \times J_n$ be the
output of the index removal method. Then 
\[
  v_n \in \co\Big\{ x_i \Bigm| i \in I_n \setminus \{ i_{1n} \} \Big\}, \quad
  w_n \in \co\Big\{ y_j \Bigm| j \in J_n \setminus \{ j_{1n} \} \Big\}.
\]
\end{proposition}

\begin{remark}
Let us note that in the case when the number of points in only one polytope is much greater than $d$ (say $\ell \gg d$),
while for the other polytope it is comparable to $d$ or even smaller than the dimension of the space, one can propose a
natural modification of the acceleration technique presented in this section. Namely, instead of shifting subpolytopes
$P_n$ and $Q_n$ in both polytopes $P$ and $Q$ one needs to shift only polytope $P_n$ inside $P$ and define 
$Q_n \equiv Q$. An analysis of such modification of Meta-algorithm~\ref{alg:PolytopesDistance} is straightforward and is
left to the interested reader.
\end{remark}

\subsection{A robust version of the meta-algorithm}
\label{subsect:RobustDistance}

Let us also present a robust version of Meta-algorithm~\ref{alg:PolytopesDistance} that takes into account finite
precision of computations and is more suitable for practical implementation than the original method. To this end,
as in Subsection~\ref{subsect:RobustNearestPoint}, suppose that instead of the ``ideal'' algorithm $\mathcal{A}$ its
``approximate'' version $\mathcal{A}_{\varepsilon}$, $\varepsilon > 0$, is given. For any two polytopes 
$V, W \subset \mathbb{R}^d$ the algorithm $\mathcal{A}_{\varepsilon}$ return an approximate (in some sense) solution of
the problem
\[
  \min_{(x, y)} \| x - y \| \quad \text{subject to} \quad x \in V, \: y \in W.
\]
To ensure finite termination of the acceleration technique based on the ``approximate'' algorithm
$\mathcal{A}_{\varepsilon}$ one obviously needs to replace the optimality conditions for the problem $(\mathcal{D})$
(see Prop.~\ref{prp:PolytopesDistOptCond}), which are used as a stopping criterion, with approximate optimality
conditions of the form
\[
  \langle v_* - w_*, x_i - v_* \rangle \ge - \eta \quad \forall i \in I, \quad
  \langle w_* - v_*, y_j - w_* \rangle \ge - \eta \quad \forall j \in J
\]
with some small $\eta > 0$. The following proposition shows how these approximate optimality conditions are related to 
approximate optimality of $(v_*, w_*)$. 

\begin{proposition} \label{prp:ApproximateOptimalityDistance}
Let $(v_*, w_*)$ be an optimal solution of the problem $(\mathcal{D})$ and a pair $(v, w) \in P \times Q$ satisfy the
inequalities
\begin{equation} \label{eq:SubOptimCondDist}
  \langle v - w, x_i - v \rangle \ge - \eta \quad \forall i \in I, \quad
  \langle w - v, y_j - w \rangle \ge - \eta \quad \forall j \in J
\end{equation}
for some $\eta > 0$. Then 
\[
  \| v - w - (v_* - w_*) \| \le \sqrt{2 \eta}, \quad
  \| v - w \| \le \dist(P, Q) + \sqrt{2\eta}. 
\]
Conversely, let $(v, w) \in P \times Q$ be such that $\| v - w - (v_* - w_*) \| \le \varepsilon$ for some 
$\varepsilon > 0$. Then the pair $(v, w)$ satisfies inequalities \eqref{eq:SubOptimCondDist} for any
$\eta > 0$ such that $\eta \ge (\diam(P) + \diam(Q) + \dist(P, Q)) \varepsilon$.
\end{proposition}

\begin{proof}
Let a pair $(v, w) \in P \times Q$ satisfy inequalities \eqref{eq:SubOptimCondDist} for some $\eta > 0$ and
$(v_*, w_*)$ be an optimal solution of the problem $(\mathcal{D})$. Observe that
\[
  \| v - w - (v_* - w_*) \|^2 = \langle v - w, v - v_* \rangle
  + \langle w - v, w - w_* \rangle 
  - \langle v_* - w_*, v - w - (v_* - w_*) \rangle.
\]
By Proposition~\ref{prp:PolytopesDistOptCond} one has
\[
  \langle v_* - w_*, x -  y - (v_* - w_*) \rangle \ge 0 \quad \forall x \in P, \: y \in Q,
\]
while from inequalities \eqref{eq:SubOptimCondDist} it obviously follows that
\[
  \langle v - w, x - v \rangle \ge - \eta \quad \forall x \in P, \quad
  \langle w - v, y - w \rangle \ge - \eta \quad \forall y \in Q.
\]
Therefore
\[
  \| v - w - (v_* - w_*) \|^2 \le 2 \eta,
\]
which yields
\[
  \Big| \| v - w \| - \| v_* - w_* \| \Big| \le \| v - w - (v_* - w_*) \| \le \sqrt{2 \eta}
\]
or, equivalently, $\| v - w \| \le \dist(P, Q) + \sqrt{2 \eta}$, since $\| v_* - w_* \| = \dist(P, Q)$.

Suppose now that $\| v - w - (v_* - w_*) \| \le \varepsilon$ for some $\varepsilon > 0$ and $(v_*, w_*)$ is an optimal
solution of the problem $(\mathcal{D})$. Adding and subtracting $v_* - w_*$ twice one gets that for any
$(x, y) \in P \times Q$ the following equality holds true:
\begin{multline*}
  \langle v - w, x - y - (v - w) \rangle =
  \langle v - w - (v_* - w_*), x - y - (v - w) \rangle 
  \\
  + \langle v_* - w_*, x - y - (v_* - w_*) \rangle
  + \langle v_* - w_*, v_* - w_* - (v - w) \rangle.
\end{multline*}
Note that the second term on the right-hand side of this equality is nonnegative by
Prop.~\ref{prp:PolytopesDistOptCond},
while the first and the last terms can be estimated as follows:
\begin{multline*}
  \big| \langle v - w - (v_* - w_*), x - y - (v - w) \rangle \big|
  \\
  \le \| v - w - (v_* - w_*) \| \Big( \| x - v \| + \| y - w \| \Big)
  \le \varepsilon \big( \diam(P) + \diam(Q) \big),
\end{multline*}
\[
  \big| \langle v_* - w_*, v_* - w_* - (v - w) \rangle \big|
  \le \| v_* - w_* \| \| v_* - w_* - (v - w) \| \le \dist(P, Q) \varepsilon.
\]
Consequently, one has
\[
  \langle v - w, x - y - (v - w) \rangle \ge - \Big( \diam(P) + \diam(Q) + \dist(P, Q) \Big) \varepsilon,
\]
and the proof is complete.
\end{proof}

To formulate an implementable robust version of Meta-algorithm~\ref{alg:PolytopesDistance}, suppose that the output 
$\mathcal{A}_{\varepsilon}(P_n, W_n)$ of the algorithm $\mathcal{A}_{\varepsilon}$ is not an approximate optimal
solution $(v_n(\varepsilon), w_n(\varepsilon))$ of the problem
\[
  \min_{(x, y)} \| x - y \| \quad \text{subject to} \quad x \in P_n, \quad y \in Q_n,
\]
but rather coefficients of the corresponding convex combinations, that is, a pair 
$(\alpha_n(\varepsilon), \beta_n(\varepsilon)) \in \mathbb{R}^{d + 1} \times \mathbb{R}^{d + 1}$ such that
\begin{gather*}
  v_n(\varepsilon) = \sum_{i \in I_n} \alpha_n^{(i)}(\varepsilon) x_i, \quad
  \sum_{i \in I_n} \alpha_n^{(i)}(\varepsilon) = 1, \quad 
  \alpha_n^{(i)}(\varepsilon) \ge 0 \quad \forall i \in I_n,
  \\
  w_n(\varepsilon) = \sum_{j \in J_n} \beta_n^{(j)}(\varepsilon) y_j, \quad
  \sum_{j \in J_n} \beta_n^{(j)}(\varepsilon) = 1, \quad 
  \beta_n^{(j)}(\varepsilon) \ge 0 \quad \forall j \in J_n.
\end{gather*}
Since the algorithm $\mathcal{A}_{\varepsilon}$ computes only an approximately optimal solution, even in the case when
$x_i$, $i \in I_n$, are affinely independent, and $y_j$, $j \in J_n$, are affinely independent, all coefficients
$\alpha_n^{(i)}(\varepsilon)$ and $\beta_n^{(j)}(\varepsilon)$ might be strictly positive. Therefore we propose to
utilize essentially the same strategy for removing indices from the sets $I_n$ and $J_n$ as is used in 
Meta-algorithm~\ref{alg:NearestPointRobust}. The main goal of this strategy is to maintain the validity of the
approximate distance decay condition
\[
  \theta_{n + 1} := \| v_{n + 1}(\varepsilon) - w_{n + 1}(\varepsilon) \| <
  \| v_n(\varepsilon) - w_n(\varepsilon) \| =: \theta_n \quad \forall n.
\]
As we will show below, this inequality guarantees finite termination of the robust meta-algorithm for finding the
distance between two polytopes given in Meta-algorithm~\ref{alg:PolytopesDistanceRobust}.

\begin{algorithm}[htbp]  \label{alg:PolytopesDistanceRobust}
\caption{Robust meta-algorithm for finding the distance between two polytopes.}

\noindent\textbf{Input:} {two collection of points 
$\{ x_1, \ldots, x_{\ell} \}, \{ y_1, \ldots, y_m \} \subset \mathbb{R}^d$, $\eta > 0$, and an algorithm
$\mathcal{A}_{\varepsilon}$, $\varepsilon > 0$, for finding the distance between two polytopes.}

\noindent\textbf{Step 0:} {Put $n = 0$, choose index sets $I_n \subseteq I$ and $J_n \subseteq J$ with 
$|J_n| = |I_n| = d + 1$, and define 
\[
  P_n = \co\{ x_i \mid i \in I_n \}, \quad Q_n = \co\{ y_j \mid j \in J_n \}.
\]
Compute $(\alpha_n(\varepsilon), \beta_n(\varepsilon)) = \mathcal{A}_{\varepsilon}(P_n, Q_n)$,
\[
  v_n(\varepsilon) = \sum_{i \in I_n} \alpha_n^{(i)}(\varepsilon) x_i, \quad
  w_n(\varepsilon) = \sum_{j \in J_n} \beta_n^{(j)}(\varepsilon) y_j, \quad 
  \theta_n = \| v_n(\varepsilon) - w_n(\varepsilon) \|.
\]
}

\noindent{\textbf{Step 1:} Compute
\[
  \rho_{xn} = \min_{i \in I} \langle v_n(\varepsilon) - w_n(\varepsilon), x_i - v_n(\varepsilon) \rangle, \quad
  \rho_{yn} = \min_{j \in J} \langle w_n(\varepsilon) - v_n(\varepsilon), y_j - w_n(\varepsilon) \rangle.
\]
If $\rho_{xn} \ge - \eta$ and $\rho_{yn} \ge - \eta$, \textbf{return} $(v_n(\varepsilon), w_n(\varepsilon))$. 
If $\rho_{xn} < - \eta$, find $i_{1n} \in I_n$ and $i_{2n} \in I$ such that
\[
  \alpha_n^{(i_{1n})}(\varepsilon) = \min_{i \in I_n} \alpha_n^{(i)}(\varepsilon), \quad
  \langle v_n(\varepsilon) - w_n(\varepsilon), x_{i_{2n}} \rangle = 
  \min_{i \in I} \langle v_n(\varepsilon) - w_n(\varepsilon), x_i \rangle,
\]
and define 
\[
  I_{n + 1} = \Big( I_n \setminus \{ i_{1n} \} \Big) \cup \{ i_{2n} \}, \quad
  P_{n + 1} = \co\Big\{ x_i \Bigm| i \in I_{n + 1} \Big\}.
\]
If $\rho_{yn} < - \eta$, find $j_{1n} \in J_n$ and $j_{2n} \in J$ such that
\[
  \beta_n^{(j_{1n})}(\varepsilon) = \min_{j \in J_n} \beta_n^{(j)}(\varepsilon), \quad
  \langle w_n(\varepsilon) - v_n(\varepsilon), y_{j_{2n}} \rangle = 
  \min_{j \in J} \langle w_n(\varepsilon) - v_n(\varepsilon), y_j \rangle,
\]
and define 
\[
  J_{n + 1} = \Big( J_n \setminus \{ j_{1n} \} \Big) \cup \{ j_{2n} \}, \quad
  Q_{n + 1} = \co\Big\{ y_j \Bigm| j \in J_{n + 1} \Big\}
\]
}

\noindent\textbf{Step 2:} {Compute 
$(\alpha_{n + 1}(\varepsilon), \beta_{n + 1}(\varepsilon)) = \mathcal{A}_{\varepsilon}(P_{n + 1}, Q_{n + 1})$,
\[
  v_{n + 1}(\varepsilon) = \sum_{i \in I_{n + 1}} \alpha_{n + 1}^{(i)}(\varepsilon) x_i, \quad
  w_{n + 1}(\varepsilon) = \sum_{j \in J_{n + 1}} \beta_{n + 1}^{(j)}(\varepsilon) y_j,
\]
and $\theta_{n + 1} = \| v_{n + 1}(\varepsilon) - w_{n + 1}(\varepsilon) \|$. If $\theta_{n + 1} < \theta_n$, set $n = n
+ 1$ 
and go to \textbf{Step 1}. Otherwise, go to \textbf{Step 3}.
}

\noindent{\textbf{Step 3:} Apply the coefficients correction method and go to Step~1.
}
\end{algorithm}

The meta-algorithm uses a heuristic rule for choosing indices $i_{1n} \in I_n$ and $j_{1n} \in J_n$ that are removed
from the subpolytopes $P_n$ and $Q_n$ on each iteration. This rule consists in finding the minimal coefficients
\[
  \alpha_n^{(i_{1n})}(\varepsilon) = \min_{i \in I_n} \alpha_n^{(i)}(\varepsilon), \quad
  \beta_n^{(j_{1n})}(\varepsilon) = \min_{j \in J_n} \beta_n^{(j)}(\varepsilon)
\]
If such choice of indices $i_{1n} \in I_n$ and $j_{1n} \in J_n$ ensures the validity of the inequality
$\theta_{n + 1} < \theta_n$ (the approximate distance decay condition), then the meta-algorithm increments $n$ and moves
to the next iteration. Otherwise, it employs the \textbf{coefficients correction method} to update 
$\alpha_n(\varepsilon)$ and $\beta_n(\varepsilon)$ in such a way that would guarantee the validity of the approximate
distance decay condition. The coefficients correction method is described below:

\begin{itemize}
\item{\textbf{Step~1:} If $\rho_{xn} \ge - \eta$, go to \textbf{Step 3}. Otherwise, find an approximately optimal
solution $\gamma_n$ of the problem
\[
  \min_{\gamma} \Big\| \sum_{i \in I_n} \gamma^{(i)} x_i - w_n(\varepsilon) \Big \|^2 
  \quad \text{subject to} \quad \sum_{i \in I_n} \gamma^{(i)} = 1.
\]
Compute $h_n = \sum_{i \in I_n} \gamma_n^{(i)} x_i$ and $\gamma_{\min} = \min_{i \in I_n} \gamma_n^{(i)}$. 
If $\gamma_{\min} < 0$, compute
\[
  \mu = \min\left\{ \frac{\alpha_n^{(i)}(\varepsilon)}{\alpha_n^{(i)}(\varepsilon) - \gamma_n^{(i)}} \Biggm| 
  i \in I_n \colon \gamma_n^{(i)} < 0 \right\},
\]
define $\alpha_n(\varepsilon) = (1 - \mu) \alpha_n(\varepsilon) + \mu \gamma_n$ and
\[
  v_n(\varepsilon) = (1 -\mu) v_n(\varepsilon) + \mu h_n, \quad
  \theta_n = \| v_n(\varepsilon) - w_n(\varepsilon) \|.
\]
If $\gamma_{\min} = 0$, define $\alpha_n(\varepsilon) = \gamma_n$, $v_n(\varepsilon) = h_n$,
$\theta_n = \| h_n - w_n(\varepsilon) \|$, and go to \textbf{Step~3}. If $\gamma_{\min} > 0$, go to \textbf{Step 2}.
}

\item{\textbf{Step 2:} Choose any $k \in I_n$, find an approximate least-squares solution $\lambda_n$ of
the system
\[
  \sum_{i \in I_n \setminus \{ k \}} \lambda^{(i)} (x_i - x_k) = 0, \quad 
  \sum_{i \in I_n \setminus \{ k \}} \lambda^{(i)} = 1
\]
and set $\lambda_n^{(k)} = - \sum_{i \in I_n \setminus \{ k \}} \lambda_n^{(i)}$. Compute
\[
  \nu = \min\left\{ - \frac{ \alpha_n^{(i)}(\varepsilon)}{\lambda_n^{(i)}} \Biggm| 
  i \in I_n \colon \lambda_n^{(i)} < 0 \right\},
  \quad \alpha_n(\varepsilon) = \alpha_n(\varepsilon) + \nu \lambda_n,
\]
and define $v_n(\varepsilon) = \sum_{i \in I_n} \alpha_n(\varepsilon) x_i$ and 
$\theta_n = \| v_n(\varepsilon) - w_n(\varepsilon) \|$.
}

\item{\textbf{Step~3:} Find an approximately optimal solution $\gamma_n$ of the problem
\[
  \min_{\gamma} \Big\| \sum_{j \in J_n} \gamma^{(j)} y_j - v_n(\varepsilon) \Big \|^2 
  \quad \text{subject to} \quad \sum_{j \in J_n} \gamma^{(j)} = 1.
\]
Compute $h_n = \sum_{j \in J_n} \gamma_n^{(i)} y_j$ and $\gamma_{\min} = \min_{j \in J_n} \gamma_n^{(j)}$. 
If $\gamma_{\min} < 0$, compute
\[
  \mu = \min\left\{ \frac{\beta_n^{(j)}(\varepsilon)}{\beta_n^{(j)}(\varepsilon) - \gamma_n^{(j)}} \Biggm| 
  j \in J_n \colon \gamma_n^{(j)} < 0 \right\},
\]
and define $\beta_n(\varepsilon) = (1 - \mu) \beta_n(\varepsilon) + \mu \gamma_n$ and
\[
  w_n(\varepsilon) = (1 -\mu) w_n(\varepsilon) + \mu h_n, \quad
  \theta_n = \| v_n(\varepsilon) - w_n(\varepsilon) \|.
\]
If $\gamma_{\min} = 0$, define $\beta_n(\varepsilon) = \gamma_n$, $w_n(\varepsilon) = h_n$,
$\theta_n = \| h_n - w_n(\varepsilon) \|$. If $\gamma_{\min} > 0$, go to \textbf{Step 4}.
}

\item{\textbf{Step 4:} Choose any $k \in J_n$, find an approximate least-squares solution $\lambda_n$ of
the system
\[
  \sum_{j \in J_n \setminus \{ k \}} \lambda^{(j)} (y_j - y_k) = 0, \quad 
  \sum_{j \in J_n \setminus \{ k \}} \lambda^{(j)} = 1
\]
and set $\lambda_n^{(k)} = - \sum_{j \in J_n \setminus \{ k \}} \lambda_n^{(j)}$. Compute
\[
  \nu = \min\left\{ - \frac{ \beta_n^{(j)}(\varepsilon)}{\lambda_n^{(j)}} \Biggm| 
  j \in J_n \colon \lambda_n^{(j)} < 0 \right\}, \quad
  \beta_n(\varepsilon) = \beta_n(\varepsilon) + \nu \lambda_n,
\]
and define $w_n(\varepsilon) = \sum_{j \in J_n} \beta_n(\varepsilon) y_j$ and 
$\theta_n = \| v_n(\varepsilon) - w_n(\varepsilon) \|$.
}
\end{itemize}

Let us present a theoretical analysis of the proposed robust meta-algorithm for computing the distance between two
polytopes. Our main goal is to show that under some natural assumptions this meta-algorithm terminates after a finite
number of steps and returns an approximately (in some sense) optimal solution of the problem $(\mathcal{D})$.

We start our analysis by showing that if on $n$th iteration of the meta-algorithm the vectors $\alpha_n(\varepsilon)$
and $\beta_n(\varepsilon)$ are such that
\begin{equation} \label{eq:ZeroCoeffsCondition}
  \min_{i \in I_n} \alpha_n^{(i)}(\varepsilon) = 0, \quad \min_{j \in J_n} \beta_n^{(j)}(\varepsilon) = 0
\end{equation}
(i.e. at least one of the coefficients of each of the corresponding convex combinations is zero), then on Step~2 of the
meta-algorithm the approximate distance decay condition $\theta_{n + 1} < \theta_n$ holds true. Therefore, the
meta-algorithm increments $n$ and moves to the next iteration without executing any other steps.

\begin{lemma} \label{lem:ApproxDistanceDecay_DistProblem}
Suppose that $\diam(P) + \diam(Q) > 0$,
\begin{multline} \label{eq:ParametersConsistency_DistProblem}
  \max\Bigg\{ 2(\diam(P) + \diam(Q) + \dist(P, Q)) \varepsilon, 
  \\
  \sqrt{\frac{\diam(P)^2 \diam(Q)^2}{\diam(P)^2 + \diam(Q)^2}}
  \sqrt{\max\{ 0, 4 \varepsilon \theta_0 - 2 \varepsilon^2 \}} \Bigg\} 
  \\
  < \eta \le 2 \min\big\{ \diam(P)^2, \diam(Q)^2 \big\}
\end{multline}
and the algorithm $\mathcal{A}_{\varepsilon}$ with $\varepsilon \ge 0$ satisfies the following \textbf{approximate
optimality condition}: for any polytopes $H = \co\{ h_1, \ldots, h_r \} \subset \mathbb{R}^d$ 
and $Z = \co\{ z_1, \ldots, z_s \} \subset \mathbb{R}^d$ one has 
\begin{align*}
  &\Big\| \sum_{i = 1}^r \alpha^{(i)} h_i - \sum_{j = 1}^s \beta^{(j)} z_j \Big\| < \dist(H, Z) + \varepsilon, 
  \\
  &\sum_{i = 1}^r \alpha^{(i)} = \sum_{j = 1}^s \beta^{(j)} = 1, \quad 
  \alpha^{(i)}, \beta^{(j)} \ge 0 \quad \forall i, j,
\end{align*}
where $(\alpha, \beta) = \mathcal{A}_{\varepsilon}(H, Z)$. Let also for some $n \in \mathbb{N}$ the stopping criterion
is not satisfied on $n$th iteration of Meta-algorithm~\ref{alg:PolytopesDistanceRobust}, equalities
\eqref{eq:ZeroCoeffsCondition} hold true on Step~1, and $\theta_{k + 1} < \theta_k$ for any 
$k \in \{ 0, 1, \ldots, n - 1 \}$. Then for $\theta_{n + 1}$ computed on Step~2 one has $\theta_{n + 1} < \theta_n$. 
\end{lemma}

\begin{proof}
\textbf{Part~1.} Let us first show that $\theta_k \ge \varepsilon$ for any $k \in \{ 0, \ldots, n \}$. Indeed, suppose
by contradiction that $\theta_k < \varepsilon$ for some $k \in \{ 0, 1, \ldots, n \}$, that is, 
$\| v_k(\varepsilon) - w_k(\varepsilon) \| < \varepsilon$. Let $(v_*, w_*)$ be an optimal solution of the problem
$(\mathcal{D})$. Then obviously $\| v_* - w_* \| \le \| v_k(\varepsilon) - w_k(\varepsilon) \| < \varepsilon$, which
yields
\[
  \| v_k(\varepsilon) - w_k(\varepsilon) - (v_* - w_*) \| \le 2 \varepsilon.
\]
Therefore by Prop.~\ref{prp:ApproximateOptimalityDistance} for any $i \in I$ and $j \in J$ one has
\begin{align*}
  \langle v_k(\varepsilon) - w_k(\varepsilon), x_i - v_k(\varepsilon) \rangle 
  &\ge - 2\big( \diam(P) + \diam(Q) + \dist(P, Q) \big) \varepsilon,
  \\
  \langle w_k(\varepsilon) - v_k(\varepsilon), y_j - w_k(\varepsilon) \rangle 
  &\ge - 2\big( \diam(P) + \diam(Q) + \dist(P, Q) \big) \varepsilon.
\end{align*}
Hence with the use of the first inequality in \eqref{eq:ParametersConsistency_DistProblem} one can conclude that 
the pair $(v_k(\varepsilon), w_k(\varepsilon))$ satisfies the stopping criterion of
Meta-algorithm~\ref{alg:PolytopesDistanceRobust} (see Step~1 of the meta-algorithm), which contradicts the assumption of
the lemma that the meta-algorithm performs $n$th iteration and the stopping criterion is not satisfied on this
iteration.

\textbf{Part~2.} Let us now prove the statement of the lemma. By our assumption the stopping criterion is not satisfied
on $n$th iteration. For the sake of shortness, below we will consider only the case when $\rho_{xn} < - \eta$ and
$\rho_{yn} < - \eta$. The proof of the cases when either $\rho_{xn} \ge - \eta$ or $\rho_{yn} \ge - \eta$ essentially
coincides with the proof of Lemma~\ref{lem:ApproxDistanceDecay}.

By condition \eqref{eq:ZeroCoeffsCondition} and the definition of indices $i_{1n} \in I_n$ and $j_{1n} \in J_n$ one has
\[
  \alpha_n^{(i_{1n})}(\varepsilon) = 0, \quad \beta_n^{(j_{1n})}(\varepsilon) = 0,
\]
which implies that
\[
  v_n(\varepsilon) \in \co\Big\{ x_i \Bigm| i \in I_n \setminus \{ i_{1n} \} \Big\}, \quad
  w_n(\varepsilon) \in \co\Big\{ y_j \Bigm| j \in J_n \setminus \{ j_{1n} \} \Big\}.
\]
Hence by the definitions of $P_{n + 1}$ and $Q_{n + 1}$ (see Step~1 of Meta-algorithm~\ref{alg:PolytopesDistanceRobust})
one has $v_n(\varepsilon) \in P_{n + 1}$ and $w_n(\varepsilon) \in Q_{n + 1}$. 

Define
\[
  x_n(t) = (1 - t) v_n(\varepsilon) + t x_{i_{2n}}, \quad
  y_n(\tau) = (1 - \tau) w_n(\varepsilon) + \tau y_{j_{2n}}.
\]
Note that $x_n(t) \in P_{n + 1}$ for any $t \in [0, 1]$ and $y_n(\tau) \in Q_{n + 1}$ for any $\tau \in [0, 1]$
due to the definitions of $P_{n + 1}$ and $Q_{n + 1}$.

For any $t \in [0, 1]$ and $\tau \in [0, 1]$ one has
\begin{align*}
  f(t, \tau) &:= \| x_n(t) - y_n(\tau) \|^2 = \| v_n(\varepsilon) - w_n(\varepsilon) \|^2
  \\
  &+ 2 t \langle v_n(\varepsilon) - w_n(\varepsilon), x_{i_{2n}} - v_n(\varepsilon) \rangle
  + 2 \tau \langle w_n(\varepsilon) - v_n(\varepsilon), y_{j_{2n}} - w_n(\varepsilon) \rangle
  \\
  &- 2 t \tau \langle v_n(\varepsilon) - x_{i_{2n}}, w_n(\varepsilon) - y_{j_{2n}} \rangle
  + t^2 \| v_n(\varepsilon) - x_{i_{2n}} \|^2 + \tau^2 \| w_n(\varepsilon) - y_{j_{2n}} \|^2.
\end{align*}
Estimating the inner products and the norms from above one gets
\begin{align*}
  f(t, \tau) &\le \theta_n^2 + 2 t \rho_{xn} + 2 \tau \rho_{yn} + 2 t \tau \diam(P) \diam(Q)
  \\
  &+ t^2 \diam(P)^2 + \tau^2 \diam(Q)^2.
\end{align*}
Let $t_* = \eta / 2\diam(P)^2$ and $\tau_* = \eta / 2\diam(Q)^2$. Note that $t_* \in (0, 1]$ and $\tau \in (0, 1]$ due
to the second inequality in \eqref{eq:ParametersConsistency_DistProblem}. Observe also that
\begin{align*}
  f(t_*, \tau_*) 
  &\le \theta_n^2 - \frac{3\eta^2}{4\diam(P)^2} - \frac{3\eta^2}{4\diam(Q)^2} + \frac{\eta^2}{2 \diam(P) \diam(Q)}
  \\
  &= \theta_n^2 - \frac{\eta^2}{2} \left( \frac{1}{\diam(P)^2} + \frac{1}{\diam(Q)^2} \right)
  - \left( \frac{\eta}{2 \diam(P)} - \frac{\eta}{2 \diam(Q)} \right)^2.
\end{align*}
Hence applying the first inequality in \eqref{eq:ParametersConsistency_DistProblem} one gets that
\begin{align*}
  \dist(P_{n + 1}, Q_{n + 1})^2 &\le \min_{t, \tau \in [0, 1]} f(t, \tau) \le f(t_*, \tau_*) 
  \\
  &\le \theta_n^2 - \frac{\eta^2}{2} \left( \frac{1}{\diam(P)^2} + \frac{1}{\diam(Q)^2} \right)
  < \theta_n^2 - 2 \theta_0 \varepsilon + \varepsilon^2,
\end{align*}
which implies that $\dist(P_{n + 1}, Q_{n + 1})^2 < (\theta_n - \varepsilon)^2$, thanks to the inequality
$\theta_0 > \theta_1 > \ldots \theta_n > \varepsilon$ that holds true by our assumption and the first part of the proof.
Hence with the use of the approximate optimality condition on algorithm $\mathcal{A}_{\varepsilon}$ one has
\[
  \theta_{n + 1} = \| v_{n + 1}(\varepsilon) - w_{n + 1}(\varepsilon) \| 
  \le \dist(P_{n + 1}, Q_{n + 1}) + \varepsilon < \theta_n,
\]  
which completes the proof.
\end{proof}

\begin{remark}
It should be noted that the lemma above holds true regardless of whether $\theta_0, \theta_1, \ldots, \theta_n$, and
$(\alpha_n(\varepsilon), \beta_n(\varepsilon))$ were computed on Step 2 or via the coefficients correction method. In
particular, it holds true even if the equalities \eqref{eq:ZeroCoeffsCondition} are satisfied for
$(\alpha_n(\varepsilon), \beta_n(\varepsilon))$ that was computed by the coefficient correctness method and not directly
computed by the algorithm $\mathcal{A}_{\varepsilon}$.
\end{remark}

With the use of the lemma above one can easily verify that if the algorithm $\mathcal{A}_{\varepsilon}$ is such that its
output $(\alpha_n(\varepsilon), \beta_n(\varepsilon)) = \mathcal{A}_{\varepsilon}(P_n, Q_n)$ always satisfies equalities
\eqref{eq:ZeroCoeffsCondition}, then Meta-algorithm~\ref{alg:PolytopesDistanceRobust} never executes the coefficients
correction method and terminates after a finite number of steps. The straightforward proof of this results is based on 
Lemma~\ref{lem:ApproxDistanceDecay_DistProblem} and in essence repeats the proof of 
Theorem~\ref{thrm:RobustMetaAlg_Wolfe}. Therefore, we omit it for the sake of shortness.

\begin{theorem}
Let $\ell, m \ge d + 1$, $\diam(P) + \diam(Q) > 0$, inequalities \eqref{eq:ParametersConsistency_DistProblem} hold true,
and the algorithm $\mathcal{A}_{\varepsilon}$ with $\varepsilon > 0$ satisfy the approximate optimality condition from
Lemma~\ref{lem:ApproxDistanceDecay_DistProblem}. Suppose also that for any polytopes
$H = \co\{ h_1, \ldots, h_r \} \subset \mathbb{R}^d$ and $Z = \co\{ z_1, \ldots, z_s \} \subset \mathbb{R}^d$ with
$r, s \ge d + 1$ there exist $i \in \{ 1, \ldots, r \}$ and $j \in \{ 1, \ldots, s \}$ such that for 
$(\alpha, \beta) = \mathcal{A}_{\varepsilon}(P, Q)$ one has $\alpha^{(i)} = \beta^{(j)} = 0$.
Then Meta-algorithm~\ref{alg:PolytopesDistanceRobust} is correctly defined, never executes Step~3 (the coefficients
correction method), terminates after a finite number of iterations, and returns a pair 
$(v_n(\varepsilon), w_n(\varepsilon)) \in P \times Q$ such that 
\[
  \| v_n(\varepsilon) - w_n(\varepsilon) - (v_* - w_*) \| \le \sqrt{2 \eta}, \quad
  \| v_n(\varepsilon) - w_n(\varepsilon) \| \le \dist(P, Q) + \sqrt{2 \eta},
\]
where $(v_*, w_*)$ is an optimal solution of the problem $(\mathcal{D})$.
\end{theorem}

Let us finally provide sufficient conditions for the correctness and finite termination of
Meta-algorithm~\ref{alg:PolytopesDistanceRobust} in the general case. These conditions largely coincide with the
corresponding condition for Meta-algorithm~\ref{alg:NearestPointRobust} for finding the nearest point in a polytope. 

\begin{theorem}
Let $\ell, m \ge d + 1$, $\diam(P) + \diam(Q) > 0$, inequalities \eqref{eq:ParametersConsistency_DistProblem} be
satisfied, and the following approximate optimality conditions hold true:
\begin{enumerate}
\item{for any polytopes $H = \co\{ h_1, \ldots, h_r \} \subset \mathbb{R}^d$ 
and $Z = \co\{ z_1, \ldots, z_s \} \subset \mathbb{R}^d$ one has 
\begin{align*}
  &\Big\| \sum_{i = 1}^r \alpha^{(i)} h_i - \sum_{j = 1}^s \beta^{(j)} z_j \Big\| < \dist(H, Z) + \varepsilon, 
  \\
  &\sum_{i = 1}^r \alpha^{(i)} = \sum_{j = 1}^s \beta^{(j)} = 1, \quad 
  \alpha^{(i)}, \beta^{(j)} \ge 0 \quad \forall i, j,
\end{align*}
where $(\alpha, \beta) = \mathcal{A}_{\varepsilon}(H, Z)$;
}

\item{if for some $n \in \mathbb{N}$ Meta-algorithm~\ref{alg:PolytopesDistanceRobust} executes Step~1 of 
the coefficient correction method, then $\sum_{i \in I_n} \gamma_n^{(i)} = 1$ and 
$\| h_n - w_n(\varepsilon) \| \le \| v_n(\varepsilon) - w_n(\varepsilon) \|$; similarly, if for some $n \in \mathbb{N}$
Meta-algorithm~\ref{alg:PolytopesDistanceRobust} executes Step~3 of the coefficient correction method, then 
$\sum_{i \in I_n} \gamma_n^{(i)} = 1$ and $\| h_n - v_n(\varepsilon) \| \le \| v_n(\varepsilon) - w_n(\varepsilon) \|$; 
}

\item{if for some $n \in \mathbb{N}$ the vectors $x_i$, $i \in I_n$, are affinely independent and 
Meta-algorithm~\ref{alg:PolytopesDistanceRobust} executes Step~1 of the coefficient correction method, then 
$\gamma_{\min} \le 0$; similarly, if for some $n \in \mathbb{N}$ the vectors $y_j$, $j \in J_n$, are affinely
independent and Meta-algorithm~\ref{alg:PolytopesDistanceRobust} executes Step~3 of the coefficient correction method,
then $\gamma_{\min} \le 0$;
}

\item{if for some $n \in \mathbb{N}$ Meta-algorithm~\ref{alg:PolytopesDistanceRobust} executes Step 2 (Step 4) 
of the coefficients correction method, then
\begin{gather*}
  \Big\| \sum_{i \in I_n \setminus \{ k \}} \gamma^{(i)} (x_i - x_k) \Big\| 
  < \frac{\theta_{n - 1} - \theta_n}{\nu}, \quad 
  \sum_{i \in I_n \setminus \{ k \}} \gamma^{(i)} \ne 0
  \\
  \bigg( \Big\| \sum_{j \in J_n \setminus \{ k \}} \gamma^{(j)} (y_j - y_k) \Big\| 
  < \frac{\theta_{n - 1} - \theta_n}{\nu}, \quad 
  \sum_{j \in J_n \setminus \{ k \}} \gamma^{(j)} \ne 0 \bigg),
\end{gather*}
where $\nu > 0$ is computed on Step~2 (Step~4) (if $n = 0$, then only the second inequality should be satisfied).
}
\end{enumerate}
Then Meta-algorithm~\ref{alg:PolytopesDistanceRobust} is correctly defined, executes the coefficients correction method
at most once per iteration, terminates after a finite number of iterations, and returns a pair 
$(v_n(\varepsilon), w_n(\varepsilon)) \in P \times Q$ such that 
\[
  \| v_n(\varepsilon) - w_n(\varepsilon) - (v_* - w_*) \| \le \sqrt{2 \eta}, \quad
  \| v_n(\varepsilon) - w_n(\varepsilon) \| \le \dist(P, Q) + \sqrt{2 \eta},
\]
where $(v_*, w_*)$ is an optimal solution of the problem $(\mathcal{D})$.
\end{theorem}

The proof of this theorem is essentially the same as the proof of Theorem~\ref{thm:RobustMetaAlgCorrectness}. That is
why we omit it for the sake of shortness.

\subsection{Numerical experiments}
\label{subsect:NumExperiments_Dist}

The acceleration technique for methods for computing the distance between two polytopes described in 
Meta-algorithm~\ref{alg:PolytopesDistanceRobust} was verified numerically for various values of $d$, $\ell$, and $m$.
Let us briefly describe the results of our numerical experiments.

We set $\ell = m$ and for each choice of $d$ and $\ell$ randomly generated 10 problems. Average computation time for 
these problem was used to assess the efficiency of the acceleration technique.

The problem data was generated similarly to the case of the nearest point problem. First, we generated $2 \ell$ points 
$\{ \widehat{x}_1, \ldots, \widehat{x}_{\ell} \}$ and $\{ \widehat{y}_1, \ldots, \widehat{y}_{\ell} \}$, uniformly 
distributed over the $d$-dimensional cube $[-1, 1]^d$. Then these points were compressed and shifted as follows
\begin{align*}
  x_i &= (1 + 0.01 \widehat{x}_i^{(1)}, \widehat{x}_i^{(2)}, \ldots, \widehat{x}_i^{(d)}) 
  \quad \forall i \in \{ 1, \ldots, \ell \},
  \\
  y_j &= (- 1 + 0.01 \widehat{y}_i^{(1)}, \widehat{y}_i^{(2)}, \ldots, \widehat{y}_i^{(d)}) 
  \quad \forall i \in \{ 1, \ldots, \ell \},
\end{align*}
so that the polytopes $P$ and $Q$ do not intersect. Numerical experiments showed that this particular problem
is especially challenging for methods for computing the distance between polytopes.

Without trying to conduct comprehensive numerical experiments, we tested the acceleration technique on 2 methods:
a modification of the MDM method for computing the distance between polytopes called ALT-MDM \cite{ALTMDM}, and the
method based on solving the quadratic programming problem 
\begin{align*}
  &\min_{(\alpha, \beta)} \:
  \frac{1}{2} \Big\| \sum_{i = 1}^{\ell} \alpha^{(i)} x_i - \sum_{j = 1}^{\ell} \beta^{(j)} y_j \Big\|^2 \quad 
  \text{subject to} 
  \\ 
  &\sum_{i = 1}^{\ell} \alpha^{(i)} = 1, \quad \alpha^{(i)} \ge 0, \quad i \in I, \quad
  \quad \sum_{j = 1}^{\ell} \beta^{(j)} = 1, \quad \beta^{(j)} \ge 0, \quad j \in J
\end{align*}
with the use of \texttt{quadprog}, the standard \textsc{Matlab} routine for solving quadratic programming problems. 
We used this routine with default settings. The inequality $\Delta_x(u_k, v_k) + \Delta_y(u_k, v_k) < 10^{-4}$ was used
as the termination criterion for the ALT-MDM method (see \cite{ALTMDM}). The number of iterations of this method was
limited to $10^6$.

Both, the ALT-MDM method and the quadratic programming method were implemented ``on their own'' and also incorporated
within the robust acceleration technique (Meta-algorithm~\ref{alg:PolytopesDistanceRobust}). The initial guess for the
meta-algorithm was chosen as
\[
  I_0 = J_0 = \{ 1, \ldots, d + 1 \}, \quad P_0 = \co\{ x_1, \ldots, x_{d + 1} \}, \quad
  Q_0 = \co\{ y_1, \ldots y_{d + 1} \}.
\]
We also set $\eta = 10^{-4}$. 

\begin{figure}[t] 
\includegraphics[width=0.5\textwidth]{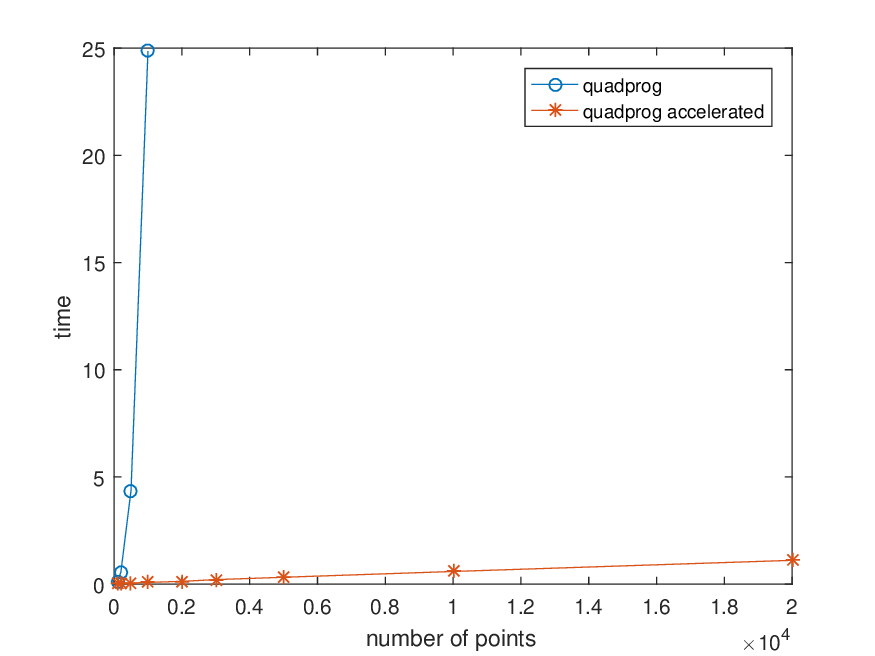}
\includegraphics[width=0.5\textwidth]{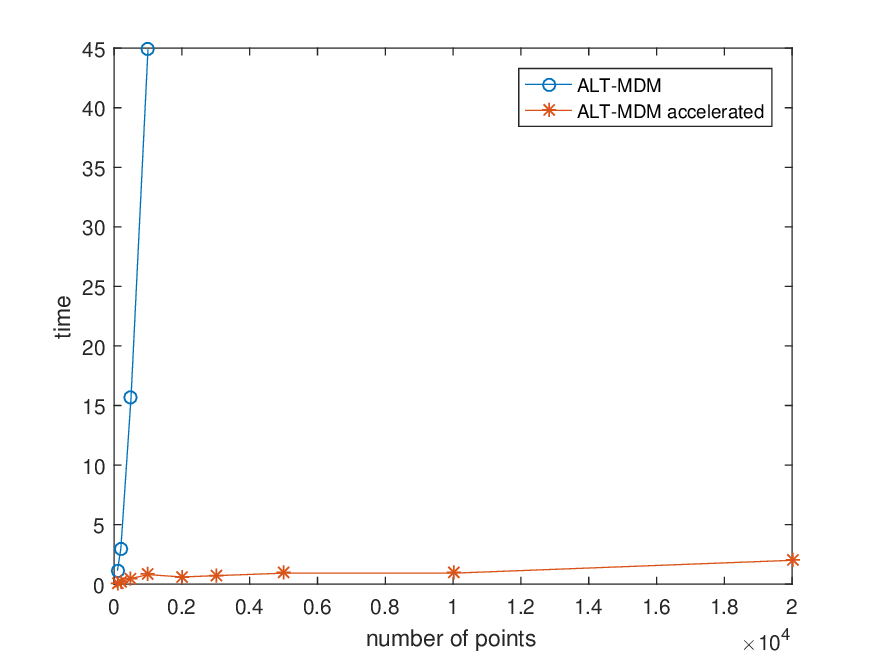}
\caption{The results of numerical experiments in the case $d = 3$ for \texttt{quadprog} routine (left figure) and
the ALT-MDM method (right figure).}
\label{fig:Dist3}
\end{figure}

\begin{figure}[t]
\includegraphics[width=0.5\textwidth]{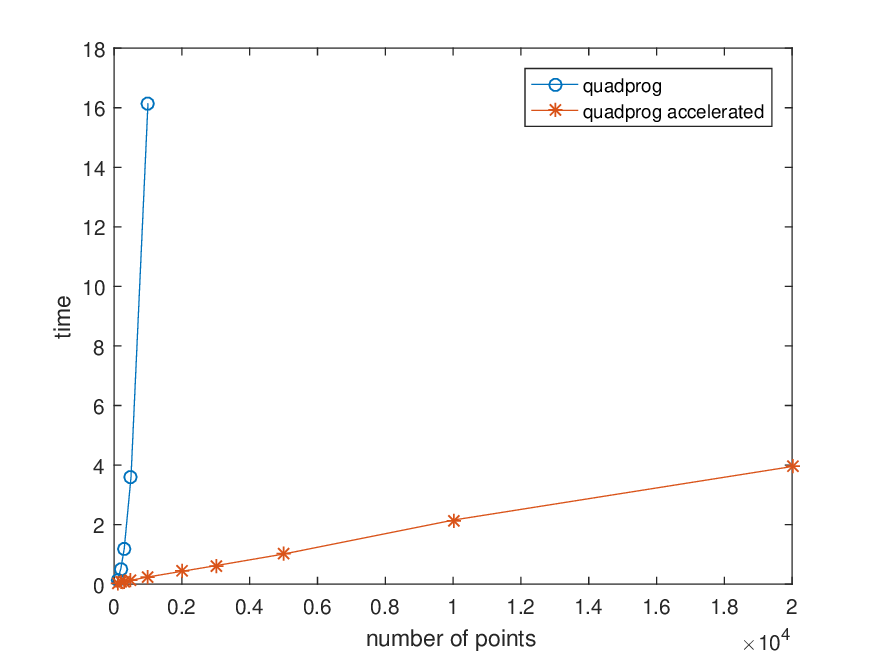}
\includegraphics[width=0.5\textwidth]{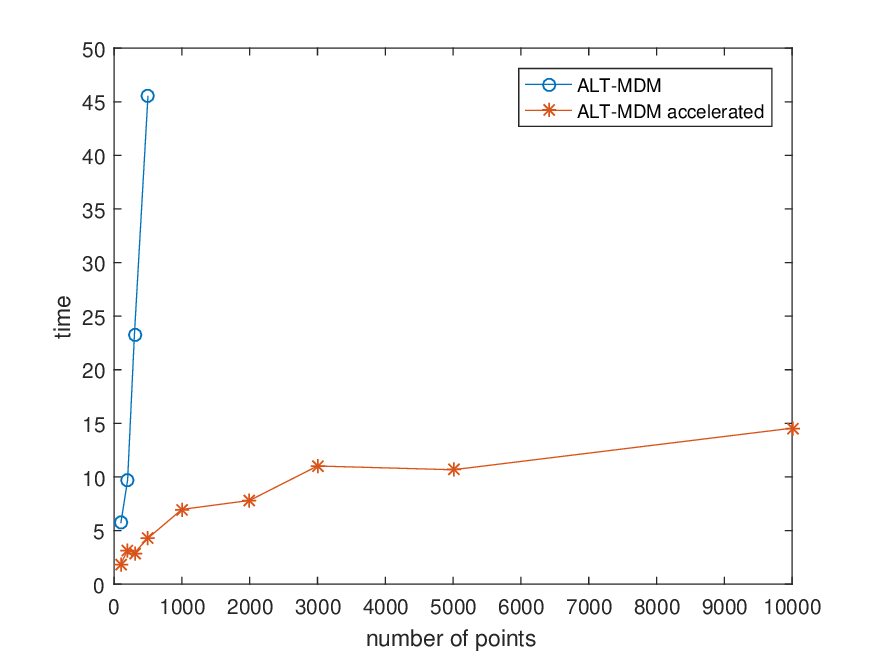}
\caption{The results of numerical experiments in the case $d = 10$ for \texttt{quadprog} routine (left figure) and
the ALT-MDM method (right figure).}
\label{fig:Dist10}
\end{figure}

% \begin{figure}[t] 
% \includegraphics[width=0.5\textwidth]{DoubleQuadProg_50}
% \caption{The results of numerical experiments in the case $d = 50$ for \texttt{quadprog} routine.}
% \label{fig:Dist50}
% \end{figure}

The results of numerical experiments are given in Fig.~\ref{fig:Dist3}--\ref{fig:Dist10}. Let us point out
that these results are qualitatively the same as the results of numerical experiments for the nearest point problem
given in Sect.~\ref{subsect:NumExperiments_NPP}. Therefore, the discussion of numerical experiments from 
Sect.~\ref{subsect:NumExperiments_NPP} is valid for Meta-algorithm~\ref{alg:PolytopesDistanceRobust} as well.

\bibliographystyle{abbrv}  
\bibliography{PolytopeDist_bibl}

\end{document}